\documentclass[a4paper]{article}

\usepackage{latexsym,amsfonts,amssymb,amsmath,amsthm}
\usepackage{epsfig,graphicx,picins,picinpar,subfigure}
\usepackage[numbers,sort&compress]{natbib}
\newtheorem{theorem}{Theorem}[section]
\newtheorem{lemma}[theorem]{Lemma}

\newtheorem{remark}[theorem]{Remark}
\newtheorem{proposition}[theorem]{Proposition}
\newtheorem{definition}[theorem]{Definition}
\numberwithin{equation}{section}

\newcommand\norm[1]{\lVert#1\rVert}
\newcommand\abs[1]{\lvert#1\rvert}

\newcommand\eps{\varepsilon}
\def\p{\partial}
\def\R{\mathbb R}

\newcommand{\RR}{\ensuremath{\mathbb{R}}}

\def\R{\mathbb R}

\def\RR{\mathbb R}
\begin{document}
\title{Group Actions on Monotone Skew-Product Semiflows
with Applications}
\author {
Feng Cao \\
Department of Mathematics\\
Nanjing University of Aeronautics and Astronautics
\\ Nanjing, Jiangsu, 210016, P. R. China\\
E-mail: fcao@nuaa.edu.cn
\\
\\
Mats Gyllenberg\\
Department of
Mathematics and Statistics\\
University of Helsinki, FIN-00014, Finland\\
E-mail: mats.gyllenberg@helsinki.fi
\\
\\
Yi Wang\thanks{Corresponding Author. Partially supported by NSF of China No.10971208, and
  the Finnish Center of Excellence
in Analysis and Dynamics.}\\
$^a$Department of Mathematics\\
 University of Science and Technology of China
\\ Hefei, Anhui, 230026, P. R. China
\\
$^b$Department of
Mathematics and Statistics\\
University of Helsinki, FIN-00014, Finland
\\
E-mail: wangyi@ustc.edu.cn\\
}
\date{}
\maketitle
\begin{abstract}
We discuss a general framework
of monotone skew-product semiflows under a connected group action. In a prior work, a
compact connected group $G$-action has been considered on a strongly monotone
skew-product semiflow. Here we relax the requirement of strong monotonicity of the skew-product
semiflows and the compactness of $G$,
and establish a theory concerning
symmetry or monotonicity properties of uniformly stable $1$-cover minimal sets. We then apply this
theory to show rotational
symmetry of certain stable entire solutions for a class of non-autonomous reaction-diffusion equations on $\RR^n$, as well as monotonicity of stable travelling
waves of some nonlinear diffusion equations in time recurrent structures including almost periodicity and
almost automorphy.
\end{abstract}

%==================section 1. Introduction========================%
\section{{\large{\bf Introduction}}}
%=========================introduction=============================
In this article, we investigate monotone skew-product semiflows with certain symmetry such as ones with respect to
rotation or translation. We will
restrict our attention to solutions which are `stable' in
a certain sense and discuss the relation between stability and
symmetry.

Historically, stability is in many cases known to imply some sort of
symmetry. For autonomous (or time-periodic) parabolic equations, any
stable equilibrium (or time-periodic) solution inherits the
rotational symmetry of the domain $\Omega$ (see \cite{CH,Ma} for
bounded domain and \cite{OM1,OM2} for unbounded domain). In
\cite{OM1,OM2}, the symmetry of the stable solutions was also
obtained for degenerate diffusion equations and systems of
reaction-diffusion equations. Ni et al.\cite{NPY} showed the
spatially symmetric or monotonic structure of stable solutions in
shadow systems as a limit of reaction-diffusion systems. It is now
well known that parabolic equations and systems admitting the
comparison principle define (strongly) monotone dynamical systems,
whose concept was introduced in \cite{Hi} (see \cite{HS,P4} for a
comprehensive survey on the development of this theory). If the
domain and the coefficients in such an equation or system exhibit a
symmetry, then the dynamical system commutes with the action of some
topological group $G$.  Extensions and generalizations of group
actions to a general framework of (strongly) monotone systems were
given by \cite{LW,MP,OM1,OM2,Ta1}.

Non-periodic and non-autonomous equations have been attracting more
attention recently. A unified framework to study non-autonomous
equations is based on the so-called skew-product semiflows (see
\cite{Se,ShYi}). In \cite{WY},  a
compact connected group $G$-action was considered on a strongly monotone
skew-product semiflow $\Pi_t$. Assuming that a minimal set $K$ of
$\Pi_t$ is stable, it was proved in \cite{WY} that $K$ is residually symmetric, and
moreover, any uniformly stable orbit is asymptotically symmetric. In
this article, motivated by Ogiwara and Matano \cite{OM1,OM2}, we
relax the restriction of strong monotonicity of the skew-product
semiflow $\Pi_t$, as well as the compactness of the acting group
$G$. To formulate our results precisely, we let $K$ be a uniformly
stable $1$-cover of the base flow. Under the assumption that $\Pi_t$
is only monotone and $G$ is only connected, we establish the
globally topological structure of the group orbit $GK$ of $K$, where
$GK=\{g\cdot(x,\omega):g\in G \mbox{ and } (x,\omega)\in K\}$ (see
Theorem B). Roughly speaking, the group orbit $GK$ either coincides
with $K$ (which entails that $K$ is $G$-symmetric); or otherwise,
$GK$ is a $1$-dimensional continuous subbundle on the base, while
each fibre of such bundle being totally ordered and homeomorphic to
$\R$. In particular, when the second case holds, the uniform stability of
$K$ will imply the asymptotic uniform stability (see Theorem D).

Our main theorems are extensions of symmetry results in \cite{OM1,OM2} on stable equilibria (resp. fixed points) for continuous-time (resp. discrete-time) monotone systems. This enables us to investigate the symmetry of certain stable entire solutions of
nonlinear reaction-diffusion equations in {\it time recurrent
structures} (see Definition \ref{recurrent}) on a symmetric domain.
This is satisfied, for instance, when the reaction term is a {\it uniformly almost periodic} or, more
generally,  a {\it  uniformly almost automorphic} function in $t$
(see Section 2 for more details).

Since strong monotonicity of the skew-product semiflow is weakened,
we are able to deal with the time-recurrent parabolic equation on an
unbounded symmetric domain such as the entire space $\RR^n$. For
non-autonomous parabolic equations, radial symmetry has been shown
to be a consequence of positivity of the solutions (see, e.g.
\cite{BS,HP,P2,P3} and references therein). For non-autonomous
parabolic equations on $\RR^n$, we also refer to a series of very
recent work by Pol\'{a}\v{c}ik \cite{P1,P2,P5} on this topic and its
applications. In particular, he \cite{P2} proved that, under some
symmetric conditions, any positive bounded entire solution decaying
to zero at spatial infinity uniformly with respect to time is
radially symmetric. However, as far as we know, symmetry properties
of certain  stable entire (possibly sign-changing) solutions of
non-autonomous parabolic equations on $\RR^n$ have been hardly
studied. By applying our abstract results mentioned above, we shall
initiate our research on this aspect. More precisely, we show that
(see Theorem \ref{theorem 4.6}) any uniformly stable entire solution
is radially symmetric, provided that it satisfies certain module
containment (see Definition \ref{modulef}) and decays to zero at
spatial infinity uniformly with respect to time.

Note also that we have relaxed the requirement of compactness of the
acting group $G$. This will allow one to discuss symmetry or
monotonicity properties with respect to translation group. Based on
this, one can investigate monotonicity of the uniformly-stable
traveling waves for time-recurrent bistable reaction-diffusion
equations or systems. Traveling waves in time-almost periodic
nonlinear evolution equations governed by bistable nonlinearities
were first established in a series of pioneer work by  Shen
\cite{Sh1}-\cite{Sh3}. In \cite{Sh1,Sh2}, she
 proved
 the existence of such almost-periodic traveling waves, and showed that any such monotone traveling wave is
uniformly-stable. By using our abstract results, on the other hand, we give a converse theorem (see Theorem \ref{theorem 4.1})
  to that of Shen's, i.e., any uniformly-stable almost-periodic traveling wave is monotone. Moreover, we shall also
  show that any uniformly-stable almost-periodic traveling wave is uniformly stable with asymptotic phase (see Theorem \ref{theorem 4.2}). The same result as Theorem \ref{theorem 4.2} can also be found in
Shen \cite{Sh1}. But our approach (by Theorem D) was introduced in a very general framework,
 and hence, it can be applied in a rather general context and to wider classes of equations
with little modification.

This paper is organized as follows.  In section 2, we present some
basic concepts and preliminary results in the theory of skew-product
semiflows and almost periodic (automorphic) functions which will be
important to our proofs. We state our main results in Section 3,
where we also give standing assumptions characterizing our general
framework. Sections 4-6 contain the proofs of our main results. In
section 7, we apply our abstract theorems to obtain symmetry
properties of certain stable entire (possibly sign-changing)
solutions of non-autonomous parabolic equations on $\RR^n$, as well
as the monotonicity of stable almost-periodic traveling waves for
time-recurrent reaction-diffusion equations.

\section{Notation and preliminary results}
In this section, we summarize some preliminary materials to be used
in later sections. First,  we summarize some lifting properties of
compact dynamical systems. We then collect definitions and basic
facts concerning monotone skew-product semiflows and
order-preserving group actions. Finally, we give a brief review
about uniformly almost periodic (automorphic) functions and flows.

Let $\Omega$ be a compact metric space with metric $d_{\Omega}$, and
$\sigma: \Omega\times \mathbb{R} \rightarrow \Omega$ be a continuous flow on $\Omega$, denoted by
$(\Omega,\sigma)$ or $(\Omega,\mathbb{R})$. As has become customary, we
denote the value of $\sigma$ at $(\omega,t)$ alternatively by $\sigma_t(\omega)$
 or $\omega\cdot t$.  By definition, $\sigma_0(\omega)=\omega$  and
$\sigma_{t+s}(\omega)=\sigma_t(\sigma_s(\omega))$ for all $t,s\in
\mathbb{R}$ and $\omega\in \Omega$. A subset $S\subset \Omega$ is
\emph{invariant} if $\sigma_{t}(S)=S$ for every $t\in \mathbb{R}$. A
non-empty compact invariant set $S\subset \Omega$ is called
\emph{minimal} if it contains no non-empty, proper and invariant
subset. We say that the continuous flow $(\Omega,\mathbb{R})$ is
\emph{minimal} if $\Omega$ itself is a minimal set.
%{\it distal} if
%$\inf\limits_{t\in
%\mathbb{R}}d_\mathcal{O}(\sigma_t(\omega_1),\sigma_t(\omega_2))>0$
%whenever $\omega_1,\omega_2\in\Omega$ and $\omega_1\neq\omega_2$.
Let $(Z,\mathbb{R})$ be another continuous flow. A continuous map
$p: Z\rightarrow \Omega$ is called a \emph{flow homomorphism} if
$p(z\cdot t)=p(z)\cdot t$ for all $z\in Z$ and $t \in \mathbb{R}$. A
flow homomorphism which is onto is called a {\it flow epimorphism}
and a one-to-one flow epimorphism is referred as {\it a flow
isomorphism}. We note that a homomorphism of minimal flows is
already an epimorphism.

We say that a Banach space $(V,\norm{\cdot})$ is {\it ordered} if it
contains a closed convex cone, that is, a non-empty closed subset
$V_+\subset V$ satisfying $V_++V_+\subset V_+$, $\alpha V_+\subset
V_+$ for all $\alpha \geq 0$, and $V_+\cap (-V_+)=\{0\}$.  The cone
$V_+$ induces an \emph{ordering} on $V$ via $x_{1}\leq x_{2}$ if
$x_{2}- x_{1}\in V_+$. We write $x_{1}< x_{2}$ if $x_{2}- x_{1}\in
V_+\setminus \{0\}$.
%The cone $V_+$ is said to be \emph{normal} if the
%norm $\|\cdot\|$ is \emph{semi-monotone}, i.e., there is a constant
%$c$ such that $0\leq x_1 \leq x_2$ implies that
%$\|x_1\|\leq c\|x_2\|$. A norm of $V$ is called monotone if $0\leq
%x_1 \leq x_2$ implies that $\|x_1\|\leq \|x_2\|$. Since the cone is
%normal, there always exists a monotone norm which is equivalent to
%the usual norm in the Banach space $V$. Thus, throughout the paper
%we will consider the Banach space $V$ with a monotone norm
%$\norm{\cdot}$.
Given $x_1,x_2\in V$, the set $[x_1,x_2]=\{x\in V:x_1\le x\le x_2\}$
is called a {\it closed order interval} in $V$, and we write
$(x_1,x_2)=\{x\in V:x_1< x< x_2\}$.

 A subset $U$ of $V$ is said to be {\it
order convex} if for any $a, b\in U$ with $a<b$, the segment
$\{a+s(b-a): s\in [0,1]\}$ is contained in $U$. And $U$ is called
{\it lower-bounded} (resp. {\it upper-bounded}) if there exists an
element $a\in V$ such that $a\leq U$ (resp. $a\geq U$). Such an $a$ is
said to be a {\it lower bound} (resp. {\it upper bound}) for $U$.
 A lower bound $a_0$ is said to be the {\it greatest
lower bound} (g.l.b.), if any other lower bound $a$ satisfies $a\leq
a_0$. Similarly, we can define the {\it least upper bound} (l.u.b.).

\vskip 2mm
Let $X=[a,b]_V$ with $a\ll b$ ($a,b\in V$) or $X=V_+$, or
furthermore, $X$ be a closed order convex subset of $V$. Throughout this paper,
we always assume that, for any $u,v\in X$, the
greatest lower bound of $\{u,v\}$, denoted by $u\wedge v$, exists
and that $(u,v)\mapsto u\wedge v$ is a continuous mapping from
$X\times X$ into $X$.

 Let $\mathbb{R}^{+}= \{t\in\mathbb{R}: t\geq 0 \}$. We
consider a continuous \emph{skew-product semiflow} $\Pi:
\mathbb{R}^{+}\times X\times \Omega  \rightarrow X\times \Omega$
defined by
\begin{equation}\label{s-p flow}
\Pi_{t}(x,\omega)=\left(u(t,x,\omega),\omega\cdot t\right),
\quad\forall (t,x,\omega)\in \mathbb{R}^{+}\times X\times \Omega,
\end{equation}
satisfying (1) $\Pi_{0}= \mathrm{Id}$; (2) the \emph{cocycle
property}: $u(t+s,x,\omega)=u\left(s,u(t,x,\omega),\omega\cdot
t\right)$, for each $(x,\omega)\in X\times \Omega$ and $s,t \in
\mathbb{R}^{+}$. A subset $A\subset X\times \Omega$ is
\emph{positively invariant} if $\Pi_{t}(A)\subset A$ for all $t\in
\mathbb{R}^{+}$; and {\it totally invariant} if $\Pi_{t}(A)=A$ for
all $t\in \mathbb{R}^{+}$. The \emph{forward orbit} of any
$(x,\omega)\in X\times \Omega$ is defined by $O^+(x,\omega)=
\left\{\Pi_{t}(x,\omega): t\geq 0 \right\}$, and the
\emph{omega-limit set} of $(x,\omega)$ is defined by
$\mathcal{O}(x,\omega)=\{(\hat{x},\hat{\omega}) \in X\times \Omega :
\Pi_{t_{n}}(x,\omega)\rightarrow (\hat{x},\hat{\omega})~
(n\rightarrow\infty) \text{\ for some sequence\ } t_{n}\rightarrow
\infty \}$. Clearly, if a forward orbit $O^+(x,\omega)$ is
relatively compact, then the omega-limit set $\mathcal{O}(x,\omega)$ is a
nonempty, compact and totally invariant subset in $X\times \Omega$
for $\Pi_t$.

%A \emph{flow extension} of $(\Omega\times X, \Pi, \mathbb{R}^{+})$
%is a continuous skew-product flow $(\Omega\times X, \hat{\Pi},
%\mathbb{R})$ such that $\Pi(t,\omega,x)=\hat{\Pi}(t,\omega,x)$ for
%each $(\omega,x)\in \Omega\times X$ and all $t\in \mathbb{R}^{+}$. A
%compact positively invariant subset is said to admit a flow
%extension if the semiflow restricted to it does. Actually, a compact
%positively invariant set $A\subset \Omega\times X$ admits a flow
%extension if every point in $A$ admits a unique backward orbit which
%remains inside the set $A$ (see \cite{ShenYi}).

Let $P:X\times \Omega\to \Omega$ be the natural projection. A
compact positively invariant set $K\subset X\times \Omega$ is called
%an {\it almost 1-cover} (or {\it almost automorphic extension}) of
%$\Omega$ if there exists $\omega_0\in \Omega$ such that
%$P^{-1}(\omega_0)\cap K$ consists of a unique element. And, $K$ is
a {\it $1$-cover} of the base flow if $P^{-1}(\omega)\cap K$
contains a unique element for every $\omega\in \Omega$. In this
case, we denote the unique element of $P^{-1}(\omega)\cap K$ by
$(c(\omega),\omega)$ and  write
$K=\{(c(\omega),\omega):\omega\in\Omega\}$, where
$c:\Omega\rightarrow X$ is continuous with
$$\Pi_t(c(\omega),\omega)=(c(\omega\cdot t),\omega\cdot t),~\forall
t\geq0,$$ and hence, $K\cap P^{-1}(\omega)=\{(c(\omega),\omega)\}$
for every $\omega\in \Omega$.
%For the sake of brevity, we hereafter
%also {\it write $K\cap P^{-1}(\omega)=(\omega,c(\omega))$ in the
%context without any confusion.

Next, we introduce some definition concerning the stability of the
skew-product semiflow $\Pi_t$. A forward orbit $O^+(x_0,\omega_0)$ of $\Pi_t$ is said to be
{\it uniformly stable} if for every $\varepsilon>0$ there is a
$\delta=\delta(\varepsilon)>0$ such that if $s\ge 0$ and
$\norm{u(s,x_0,\omega_0)-x}\le \delta(\varepsilon)$ for certain $x\in X$, then for
each $t\ge 0,$ $\norm{u(t+s,x_0,\omega_0)-u(t,x,\omega_0\cdot s)}<\varepsilon.$ The following definition is on
the uniform stability for a compact positively invariant set $K\subset X\times \Omega$:

\begin{definition}[Uniform stability for $K$]\label{stability}
A compact positively invariant set $K$ is said to be {\it uniformly
stable} if for any $\varepsilon>0$ there exists a
$\delta(\varepsilon)>0$, called the {\it modulus of uniform
stability}, such that, if $(x,\omega)\in K,~(y,\omega)\in
X\times\Omega$ are such that $\norm{x-y}\le \delta(\varepsilon)$,
then
$$\norm{u(t,x,\omega)-u(t,y,\omega)}<\varepsilon \textnormal{ for
all }t\ge 0.$$
\end{definition}

\begin{remark}\label{remark2.2}
\textnormal{It is easy to be expected that all the trajectories in a uniformly
stable set are uniformly stable. Conversely, if a trajectory has
uniformly stable property, its omega-limit set inherits it: that is,
if $O^+(x_0,\omega_0)$ is relatively compact and uniformly stable,
then the omega-limit set $\mathcal{O}(x_0,\omega_0)$ is a uniformly
stable set with
 the same modulus of uniform
stability as that of $O^+(x_0,\omega_0)$ (see \cite{NOS,Se}).}
\end{remark}

The following Lemma is due to Novo et al \cite[Proposition 3.6]{NOS}:

\begin{lemma}\label{Novo-Ob}
Assume that $(\Omega,\mathbb{R})$ is minimal. Let $O^+(x,\omega)$ be
a forward orbit of $\Pi_t$ which is relatively compact. If its
omega-limit set $\mathcal{O}(x,\omega)$ contains a minimal set $K$
which is uniformly stable, then $\mathcal{O}(x,\omega)=K$.
\end{lemma}

\vskip 1mm For skew-product semiflows, we always use the order
relation on each fiber $P^{-1}(\omega)$. We write
$(x_1,\omega)\le_\omega (<_\omega)\, (x_2,\omega)$ if $x_1\le x_2$
($x_1<x_2$). Without any confusion, we will drop the subscript
``$\omega$". One can also define similar definitions and notations
in $P^{-1}(\omega)$ as in $X$, such as order-intervals, the greatest
lower bound, the least upper bound, etc.

Let $A,B$ be two compact subsets of $X$. We define their Hausdorff
metric $$d_H(A, B)=\max\{\sup_{x\in A}d(x,B),\sup_{y\in
B}d(y,A)\},$$ where $d(x,B)=\inf\limits_{y\in B}\norm{x-y}$. We can also define
the Hausdorff metric $d_{H,\omega}(A(\omega), B(\omega))$ for any two compact subset $A(\omega)$, $B(\omega)$
of $P^{-1}(\omega)$. Again without any confusion, we drop the subscript
``$\omega$" and write $d_{H,\omega}(A(\omega), B(\omega))$ as $d_{H}(A(\omega), B(\omega))$ in the context.

Let
$K_1, K_2$ be two positively invariant compact subsets of $X\times
\Omega$. We write $K_1\prec_r K_2$ if and only if for any
$(x,\omega)\in K_1$, there exists some $(y,\omega)\in K_2$ such that
$(x,\omega)<_r (y,\omega)$, and for any $(y,\omega)\in K_2$, there
exists some $(x,\omega)\in K_1$ such that $(x,\omega)<_r
(y,\omega)$, where $\prec_r$ (resp. $<_r$) represents $\preceq$
(resp. $\leq$) or $\prec$ (resp. $<$). $K_1\succ_r K_2$ is similarly
defined.
For such $K_1,K_2\subset X\times \Omega$, the Hausdorff
distance between $K_1$ and $K_2$ is defined as
$$d(K_1,K_2)=\sup_{\omega\in \Omega}d_H(K_1(\omega),K_2(\omega)),$$ where $d_H$ is the
Hausdorff metric for compact subsets in $P^{-1}(\omega)$.

%===========definition of monotone s-p semiflow=============
\begin{definition}\label{monotone}
{\upshape The skew-product semiflow $\Pi$ is {\it monotone} if
$$\Pi_{t}(x_1,\omega)\leq \Pi_{t}(x_2,\omega)$$ whenever $(x_1,\omega)\leq (x_2,\omega)$
and $t\geq0$.}
\end{definition}

Let $G$ be a metrizable topological group with unit element $e$. We
say that $G$ acts on the ordered space $X$ if there exists a
continuous mapping $\gamma:G\times X\rightarrow X$ such that
$a\mapsto\gamma(a,\cdot)$ is a group homomorphism of $G$ into
Hom($X$), the group of homeomorphisms of $X$ onto itself. For
brevity, we write $\gamma(a,x)=ax$ for $x\in X$ and identify the
element $a\in G$ with its action $\gamma(a,\cdot)$. A group action
$\gamma$ is said to be {\it order-preserving} if, for each $a\in G$,
the mapping $\gamma(a,\cdot): X\rightarrow X$ is increasing, i.e.
$x_1\leq x_2$ in $X$ implies $ax_1\leq ax_2$. We say that $\gamma$
{\it commutes} with the skew-product semiflow $\Pi$ if
$$au(t,x,\omega)=u(t,ax,\omega), \mbox{~for any~}(x,\omega)\in X\times \Omega,t\geq 0\mbox{ and }a\in G.$$
For $x\in X$ the {\it group orbit} of $x$ is the set $Gx=\{ax:a\in G\}$. A
point $(x,\omega)\in X\times \Omega$ is said to be {\it symmetric}
if $(Gx,\omega)=\{(x,\omega)\}$.

Due to the commutative property of $G$ with $\Pi_t$, one has the following direct lemma:
\begin{lemma}\label{lemma3.6} For any $(x_0,\omega_0)\in
X\times\Omega$ and $g\in G$, its omega-limit set
$\mathcal{O}(x_0,\omega_0)$ satisfies
$$g\mathcal{O}(x_0,\omega_0)=\mathcal{O}(gx_0,\omega_0),$$ where
$g\mathcal{O}(x_0,\omega_0)=\{(gx,\omega):(x,\omega)\in\mathcal{O}(x_0,\omega_0),\omega\in\Omega\}.$
\end{lemma}
\begin{proof} Fix any $\omega\in\Omega$. Then for any $(x,\omega)\in\mathcal{O}(x_0,\omega_0)$, there
exists a sequence $\{t_n\}\rightarrow\infty$ such that
$\Pi_{t_n}(x_0,\omega_0)=(u(t_n,x_0,\omega_0),\omega_0\cdot t_n)\rightarrow(x,\omega)$
 as $n\rightarrow\infty.$
So for any $g\in G$, we have
$u(t_n,gx_0,\omega_0)=gu(t_n,x_0,\omega_0)\rightarrow gx$ as
$n\rightarrow\infty,$ and hence
$(gx,\omega)\in\mathcal{O}(gx_0,\omega_0)\cap P^{-1}(\omega).$ Therefore,
$g\mathcal{O}(x_0,\omega_0)\subset\mathcal{O}(gx_0,\omega_0).$

Conversely, for any $(y,\omega)\in\mathcal{O}(gx_0,\omega_0)$,
choose a sequence $\{s_n\}\rightarrow\infty$ such that
$\Pi_{s_n}(gx_0,\omega_0)=(u(s_n,gx_0,\omega_0),\omega_0\cdot
s_n)\rightarrow(y,\omega)$ as $n\rightarrow\infty.$ Thus,
$gu(s_n,x_0,\omega_0)=u(s_n,gx_0,\omega_0)\rightarrow y$ as
$n\rightarrow\infty$. Without loss of generality, we may assume that
$u(s_n,x_0,\omega_0)\rightarrow x$ as $n\rightarrow\infty$.
Therefore, $(x,\omega)\in\mathcal{O}(x_0,\omega_0)~\mbox{ and
}y=gx,$ which implies that $(y,\omega)\in
g\mathcal{O}(x_0,\omega_0).$ So we have proved
$\mathcal{O}(gx_0,\omega_0)\subset g\mathcal{O}(x_0,\omega_0).$ By
the arbitrariness of $\omega\in\Omega$, we directly derive the
result.\end{proof}

We finish this section with the definitions of almost periodic
(automorphic) functions and flows.

A function $f\in C(\R,\R^n)$ is {\it almost periodic} if, for any
$\varepsilon>0$, the set
$T(\varepsilon):=\{\tau:\abs{f(t+\tau)-f(t)}<\varepsilon,\,\forall
t\in \RR\}$ is relatively dense in $\RR$. $f$ is {\it almost
automorphic} if for any $\{t_n'\}\subset \RR$ there is a subsequence
$\{t_n\}$ and a function $g:\RR\to\RR^{n}$ such that $f(t+t_n)\to
g(t)$ and $g(t-t_n)\to f(t)$ hold pointwise.

 Let $D$ be a subset of
$\RR^m$. A continuous function $f:\RR\times D\to\RR^n;(t,u)\mapsto
f(t,u),$ is said to be {\it admissible} if $f(t,u)$ is bounded and
uniformly continuous on $\RR\times K$ for any compact subset
$K\subset D$.  A function $f\in C(\RR\times D,\RR^n)(D\subset
\RR^m)$ is {\it uniformly almost periodic (automorphic) in $t$}, if
$f$ is both admissible and almost periodic (automorphic) in $t\in
\RR$.

Let $f\in C(\RR\times D,\RR^n) (D\subset \RR^m)$ be admissible. Then
$H(f)={\rm cl}\{f\cdot\tau:\tau\in \RR\}$ is called the {\it hull of
$f$}, where $f\cdot\tau(t,\cdot)=f(t+\tau,\cdot)$ and the closure is
taken under the compact open topology. Moreover, $H(f)$ is compact
and metrizable under the compact open topology. The time translation
$g\cdot t$ of $g\in H(f)$ induces a natural flow on $H(f)$.
\begin{definition}\label{recurrent}
{\rm An admissible function $f\in C(\RR\times D,\RR^n)$ is called
{\it time recurrent} if $H(f)$ is minimal.}
\end{definition}
$H(f)$ is always minimal if $f$ is uniformly almost periodic
(automorphic) in $t$. Moreover, $H(f)$ is an almost periodic
(automorphic) minimal flow when $f$ is a uniformly almost periodic
(automorphic) function in $t$ (see, e.g. \cite{Se,ShYi}).

Let $f\in C(\RR\times D,\RR^n)$ be uniformly almost periodic (automorphic),
and
\begin{equation}\label{hull}
f(t,x)\sim \sum_{\lambda\in\RR}a_\lambda(x)e^{i\lambda t}
\end{equation}
 be a Fourier series of $f$ (see \cite{ShYi,Ve} for the definition and the
existence of Fourier series). Then
$\mathcal{S}=\{\lambda:a_{\lambda}(x)\not\equiv 0\}$ is called the
Fourier spectrum of $f$ associated to the Fourier series
(\ref{hull}).
\begin{definition}\label{modulef}
 $\mathcal{M}(f)=$ the smallest additive subgroup
of $\RR$ containing $\mathcal{S}(f)$ is called the {\it frequency
module of $f$}.
\end{definition}
% Moreover, $\mathcal{M}(f)$ is a countable subset of $\RR$.
 Let $f,g\in C(\RR\times \RR^n,\RR^n)$ be two uniformly almost
periodic (automorphic) functions in $t$. We have the module
containment $\mathcal{M}(f)\subset \mathcal{M}(g)$ if and only if
there exists a flow epimorphism from $H(g)$ to $H(f)$ (see,
\cite{Fi} or \cite[Section 1.3.4]{ShYi}). In particular,
$\mathcal{M}(f)=\mathcal{M}(g)$ if and only if the flow $(H(g),\R)$
is isomorphic to the flow $(H(f),\R)$.

\section{Main results}

In this section our standing assumptions are as follows: \vskip 2mm
{\bf (A1)} $\Omega$ is minimal;

{\bf (A2)} $G$ is a connected group acting on $X$ in such a way that
its action is order-preserving;

{\bf (A3)} $G$ commutes with the monotone skew-product semiflow
$\Pi_t$.  \vskip 2mm

In what follows we will denote by
$K$ a minimal set of $\Pi_t$ in $X\times \Omega$, which is a uniformly stable
1-cover of $\Omega$. In the
 context, we also write $K=\{(\bar{u}_\omega,\omega):\omega\in \Omega\}$, and $gK=\{(g\bar{u}_\omega,\omega):\omega\in
\Omega\}$ if an element $g\in G$ acts on $K$. The {\it group orbit of $K$} is defined as $$GK=\{(g\bar{u}_\omega,\omega)\in X\times \Omega:g\in G\mbox{ and }\omega\in \Omega\}.$$
We will investigate the topological structure of $GK$ in this paper.

For $\delta>0$, we define a $\delta$-neighborhood of $K$ in $X\times \Omega$:
$$B_{\delta}(K)=\{(u,\omega)\in X\times \Omega:\norm{u-\bar{u}_\omega}<\delta\}.$$
%and
%$$B_{\delta}(GK)=\{(u,\omega)\in X\times \Omega:\norm{u-g\bar{u}_\omega}<\delta\textnormal{ for
%some }g\in G\}.$$
 Hereafter, we impose the following additional condition on $K$:
 \vskip 3mm

{\bf (A4)} There exists a $\delta>0$ such that
\begin{description}
\item[{\rm (i)}] the forward orbit
$O^+(x_0,\omega_0)$ is relatively compact for any $(x_0,\omega_0)\in B_{\delta}(K)$; and moreover,
\item[{\rm (ii)}] if the $\omega$-limit set
$\mathcal{O}(x_0,\omega_0)\subset B_\delta(K)$ and $\mathcal{O}(x_0,\omega_0)\prec
hK$ (resp. $\mathcal{O}(x_0,\omega_0)\succ hK$) for some $h\in G$,  then there is
a neighborhood $B(e)\subset G$ of $e$ such that
$\mathcal{O}(x_0,\omega_0)\prec ghK$ (resp. $\mathcal{O}(x_0,\omega_0)\succ
ghK$) for any $g\in B(e)$.
\end{description}

\begin{remark}
\textnormal{In the case where $\Pi_t$ is strongly monotone, (A4-ii) is
automatically satisfied. Recall that $\Pi_t$ is strongly monotone if
$\Pi_t(x_1,\omega)\ll \Pi_t(x_2,\omega)$ whenever
$(x_1,\omega)<(x_2,\omega)$ and $t>0$ (see \cite{ShYi}). To derive
(ii) of (A4), note that the total invariance of
$\mathcal{O}(x_0,\omega_0)$ implies that, for any $(x,\omega)\in
\mathcal{O}(x_0,\omega_0)$, there exists a neighborhood
$B_{(x,\omega)}(e)\subset G$ of $e$ such that $(x,\omega)\prec ghK$
for any $g\in B_{(x,\omega)}(e)$. Considering that
$\mathcal{O}(x_0,\omega_0)$ is compact, one can find a neighborhood
$B(e)\subset G$ such that $\mathcal{O}(x_0,\omega_0)\prec ghK$ for
any $g\in B(e)$.}
\end{remark}
\begin{remark}
\textnormal{For continuous-time (discrete-time) monotone systems, assumption
(A4) was first imposed by Ogiwara and Matano \cite{OM1,OM2} to
investigate the monotonicity and convergence of the stable
equilibria (fixed points). We here give a general version in
non-autonomous cases. At first glance,
 one can observe that (A4) is just a local dynamical hypothesis nearby $K$. Accordingly, it should only yield a local total-ordering property of the group orbit $GK$ nearby $K$ (see Lemma A below). However, in what follows, we can see that it will surprisingly imply a globally topological characteristic of the whole group orbit $GK$ (see Theorem B below), which is our main result in this paper.}
\end{remark}

\noindent {\bf Lemma A}~(Local ordering-property of $GK$ nearby
$K$){\bf .} {\it Assume that {\rm (A1)-(A3)} hold. Let $K$ be a
uniformly stable $1$-cover of $\Omega$ and satisfies {\rm (A4)}.
Then there exists a neighborhood $B(e)\subset G$ of $e$ such that
$gK\preceq K\mbox{ or }gK\succeq K$, for any $g\in B(e)$.} \vskip
2mm

\noindent{\bf Theorem B}~(Global topological structure of $GK$){\bf .}
{\it Assume that {\rm (A1)-(A3)} hold and $G$
is locally compact. Let $K$ be a uniformly stable $1$-cover of $\Omega$ and satisfies {\rm (A4)}.
Then either of the following alternatives holds:
\begin{description}
\item[\rm (i)]  $GK=K$, i.e., $K$ is $G$-symmetric;

\item[\rm (ii)] There is a continuous bijective mapping $H: \Omega\times
\RR\to GK$ satisfying:

{\rm (a)} For each $\alpha\in \RR$, $H(\Omega,\alpha)=gK$
for some $g\in G$;

{\rm (b)} For each $\omega\in \Omega$, $H(\omega,\RR)=G\bar{u}_\omega$;

{\rm (c)} $H$ is
strictly order-preserving with respect to $\alpha\in \RR$, i.e.,
$$H(\omega,\alpha_1)\ll H(\omega,\alpha_2)$$ for any
 $\omega\in \Omega$ and any $\alpha_1,\alpha_2\in \RR$ with
 $\alpha_1<\alpha_2$.
 \end{description}}
\vskip 3mm
\begin{remark}
\textnormal{Roughly speaking, Theorem B implies the following dichotomy: either
$K$ is $G$-symmetric; or otherwise, its group orbit $GK$ is a
$1$-dimensional continuous subbundle on the base, while each fibre
of such bundle being totally ordered and homeomorphic to $\R$.}
\end{remark}

Based on Theorem B, one can further deduce the following two useful theorems on symmetry of $K$, as well as its uniform stability with asymptotic phase.
\vskip 2mm
\noindent {\bf Theorem C.} {\it Assume all the hypotheses in Theorem B are satisfied. If $G$ is a compact group, then $K$ is $G$-symmetric.}
\vskip 2mm

\vskip 2mm
\noindent {\bf Theorem D {\rm (Uniform stability of $K$ with asymptotic phase)}.} {\it Assume all the hypotheses in Theorem B are satisfied. If $GK\ne K$, then
there is a $\delta_*\in (0,\delta)$ such that, if $(u,\omega)\in B_{_{\delta_*}}(K)$, then its $\omega$-limit set
$\mathcal{O}(u,\omega)=hK$ for some $h\in G$. Moreover,
$$\norm{u(t,u,\omega)-h\bar{u}_{\omega\cdot t}}\to 0,\quad\textnormal{ as }t\to \infty.$$}

\section{Globally topological structure of $GK$}
In this section, we shall prove Theorems B and C under the assumption that the conclusion of Lemma A holds already. The proof of Lemma A will be given in Section 6. We first proceed to the following useful proposition.
\begin{proposition}\label{theorem 3.7}
For any $g\in G$, there exists a neighborhood $V_g\subset G$ of $g$
such that $V_gK$ is totally-ordered, i.e., $$g_1K\preceq g_2K~\mbox{
or }~g_1K\succeq g_2K,~\forall g_1,g_2\in V_g.$$
\end{proposition}
\begin{proof} Since the group $G$ is metrizable, one can write $B(e)$ in
Lemma A as $B(e)=\{g\in
G:\rho(g,e)<\delta\}$ for some $\delta>0$, where $\rho$ denotes the right-invariant
metric on $G$ (cf. \cite[Section 1.22]{MZ}) satisfying
$\rho(g\sigma,h\sigma)=\rho(g,h)$ for all $g,h,\sigma\in G.$ Thus
for any $g_1,g_2\in G$, it follows from (A2) and Lemma A
that
\begin{equation}\label{equ.3.14}g_2K\preceq g_1K\mbox~{ or
}~g_2K\succeq g_1K,~ \mbox{ whenever }~\rho(g_1^{-1}g_2,e)<\delta.
\end{equation} Now for any $g\in G$, let $V_g=\{h\in
G:\rho(g^{-1},h^{-1})<\frac{\delta}{2}\}$. It is not difficult to see that $V_g$ is a neighborhood of $g$.
 Hence if $g_1,g_2\in V_g$, then
\begin{eqnarray*}
\rho(g_1^{-1}g_2,e)&\leq& \rho(g_1^{-1}g_2,g^{-1}g_2)+\rho(g^{-1}g_2,e)\\
&= & \rho(g_1^{-1}g_2,g^{-1}g_2)+\rho(g^{-1}g_2,g_2^{-1}g_2)\\
&=&\rho(g_1^{-1},g^{-1})+\rho(g^{-1},g_2^{-1})<\delta,
\end{eqnarray*} because $\rho$ is right-invariant. As a consequence, (\ref{equ.3.14}) implies that
$$g_1K\preceq g_2K~\mbox{ or }~g_1K\succeq g_2K,~\forall
g_1,g_2\in V_g.$$ This completes the proof.
\end{proof}

%\vskip 3mm Now we state our main result on the global structure of
%$GK=\{gK:g\in G\}$.

%\begin{theorem}[Topological structure of union of $GK$]\label{theorem 3.8}
%Assume that {\rm (A1)-(A3)} hold and $G$
%is locally compact. Then either of the following alternatives holds:
%
%{\rm (i)} $GK=K$, i.e., $K$ is $G$-symmetric;
%
%{\rm (ii)} There is a continuous bijective mapping $H: \Omega\times
%\RR\to GK$ satisfying:
%
%\begin{description}
%\item{\rm (a)} For each $\alpha\in \RR$, $H(\Omega,\alpha)=gK$
%for some $g\in G$;
%
%\item{\rm (b)} For each $\omega\in \Omega$, $H(\omega,I)=G\bar{u}_\omega$;
%
%\item{\rm (c)} $H$ is
%strictly order-preserving with respect to $\alpha\in I$, i.e.,
%$h(\omega,\alpha_1)\ll h(\omega,\alpha_2)$ for any
% $\omega\in \Omega$ and any $\alpha_1,\alpha_2\in I$ with
% $\alpha_1<\alpha_2$.
% \end{description}
%Moreover, if $G$ is compact then only {\rm (i)} holds.
%
%\end{theorem}
Now we are in position to prove our main result Theorem B:
\vskip 2mm

\noindent {\it Proof of Theorem B: }  For any two $g_1,g_2\in G$, we
write $g_1\le g_2$ whenever $g_1K\preceq g_2K$. Then a partial order
``$\le$" is induced in $G$. A subset $S\subset G$ is called
totally-ordered if any two distinct elements of $S$ are related.

We first claim that $G$ is totally-ordered. To prove this, we
 define $$\mathcal{F}=\{S\subset G:S\mbox{ is connected and totally-ordered}\}.$$
By virtue of Lemma A,  $V_g\in \mathcal{F}\ne \emptyset$. Note that
$(\mathcal{F},\subset)$ is a partially-ordered set. It follows from
Zorn's lemma that $\mathcal{F}$ possesses a maximal element, say
$M$. We first show that $M$ is a closed subset of $G$. Consider the
closure $\bar{M}$ of $M$. Clearly, $\bar{M}$ is connected. Now, for
any $h_1,h_2\in\bar{M}$, there exist sequences
$\{g_n^{1}\},\{g_n^{2}\}\subset M$ such that $g_n^{1}\rightarrow
h_1,~g_n^{2}\rightarrow h_2$ as $n\rightarrow\infty.$ For each $n\in
\mathbb{N}$, $g_n^{1}\leq g_n^{2}$ or $g_n^{1}\geq g_n^{2},$ because
$M$ is totally-ordered. By taking a subsequence $\{n_k\}$, if
necessary, we obtain
$$g_{n_k}^{1}\leq g_{n_k}^{2},~\forall k\in\mathbb{N}
~\mbox{ or }~g_{n_k}^{1}\geq
g_{n_k}^{2},~\forall k\in\mathbb{N}.$$ Letting
$k\rightarrow\infty$ in the above, one has
$h_1\leq h_2$ or $h_1\geq h_2,$ because the order ``$\le$" is closed. Hence
$\bar{M}$ is totally-ordered. By the maximality of $M$, we get $M=\bar{M}$, which implies that
$M$ is closed.

In order to show that $M$ is also an open subset of $G$, we notice that for any $g\in M$,
by Proposition \ref{theorem 3.7}, there is a neighborhood $V_g\subset G$ of $g$ such that $V_g$ is
totally-ordered and connected. Suppose that $M$ is not open. Then one can find some $g\in M$ and
a sequence $\{g_n\}_{n=1}^\infty\subset V_g\setminus M$ such that $g_n\rightarrow g$ as $n\to \infty$.
Since $V_g$ is totally-ordered, we may also assume without loss of generality that
$g_n>g$ for all $n\in\mathbb{N}.$ Fix each $n\in\mathbb{N}$, we define $$W_n^+=\{h\in M\cap
V_g:h\geq g_n\}~\mbox{ and }~W_n^-=\{h\in M\cap V_g:h\leq g_n\}.$$
 A direct examination yields that (i) $M\cap V_g=W_n^+\cup W_n^-$; (ii) $W_n^+\cap W_n^-=\emptyset$ (Since $g_n\notin M$); (iii) $W_n^-\neq\emptyset$ (Since $g\in W_n^-$); and
(iv) $W_n^+,W_n^-$ are closed in $M\cap V_g$.
%(Indeed, let $\{h_k\}\subset W_n^+$ (resp. $W_n^-$)
%satisfying $h_k\rightarrow h\in M\cap V_g$ as $k\rightarrow\infty$.
%Then by $h_k\geq g_n$ (resp. $h_k\leq g_n$) we have $h\geq g_n$
%(resp. $h\leq g_n$));
By the connectivity of $M\cap V_g$, we have $W_n^+=\emptyset$,
and hence $W_n^-=M\cap V_g$. Since $g_n\notin M$, it entails
that $M\cap V_g<g_n$ for each $n\in\mathbb{N}.$ Letting $n\rightarrow\infty$, we therefore
obtain
\begin{equation}\label{equ.3.16}M\cap V_g\leq g.\end{equation}
Furthermore, {\it we assert that $M\leq g.$} Otherwise, noticing
that $g\in M$ and $M$ is totally-ordered,  there is an $f\in M$ such
that $f>g$. Since $M$ is also connected and locally compact, it
follows from  \cite[Appendix, Proposition Y1, Page 434]{OM1} that
there is an order-preserving homeomorphism
$$\tilde{h}:[g,f]_M=\{h\in M:g\leq h\leq f\}\rightarrow[0,1]$$ with
$\tilde{h}(g)=0$ and $\tilde{h}(f)=1$. Thus by choosing $g_*\in\tilde{h}^{-1}(\delta)$
 with $0<\delta\ll1$, one has $g_*\in
(V_g\cap M)\setminus\{g\}$ and $g_*>g$, which is a contradiction to
(\ref{equ.3.16}). Thus we have proved the assertion.

 On the other
hand, recall that $g_n\in V_g$ and $g_n>g$ for every
$n\in\mathbb{N}$. Now we fix some $g_n$. Since $V_g$ is connected,
totally-ordered, and locally compact, \cite[Appendix, Proposition
Y1, Page 434]{OM1} again implies that there is an order-preserving
homeomorphism
$$\hat{h}:[g,g_n]_{_{V_g}}=\{h\in V_g:g\leq h\leq g_n\}\rightarrow[0,1]$$ with
$\hat{h}(g)=0$ and $\hat{h}(g_n)=1$. Let
$\hat{M}=M\cup[g,g_n]_{_{V_g}}$. Then  $\hat{M}\supsetneq M$. Due to the assertion in the above paragraph,
we obtain that $\hat{M}$ is connected and totally-ordered. This
contradicts the maximality of $M$. Accordingly, $M$
is an open subset of $G$.

Since $M$ is both open and closed in $G$, it follows from the connectivity of $G$ that
$G=M$. Thus we have proved the claim that $G$ is totally-ordered.

\vskip 3mm
Based on this claim, precisely one of
the following three alternatives must occur:

$\qquad$(Alt$_a$) The least upper bound (l.u.b.) of $G$ exists;

$\qquad$(Alt$_b$) The greatest lower bound (g.l.b.) of $G$ exists;

$\qquad$(Alt$_c$) Neither l.u.b. nor g.l.b. of $G$ exists.

\vskip 2mm
If (Alt$_a$) holds, then one can find a $g_0\in G$ such that
$$g\bar{u}_\omega\leq g_0\bar{u}_\omega~~\mbox{ for any }~\omega\in \Omega\mbox{ and }g\in G.$$ In particular,
$g_0^{2}\bar{u}_\omega\leq g_0\bar{u}_\omega$, and hence
$g_0\bar{u}_\omega=g_0^{-1}(g_0^{2}\bar{u}_\omega)\leq g_0^{-1}(g_0\bar{u}_\omega)=\bar{u}_\omega\leq g_0\bar{u}_\omega,$
which entails that $g_0\bar{u}_\omega=\bar{u}_\omega$ for any $\omega\in \Omega$. Consequently,
$g^{-1}\bar{u}_\omega\leq \bar{u}_\omega$, and hence
$\bar{u}_\omega=g(g^{-1}\bar{u}_\omega)\leq g\bar{u}_\omega\leq \bar{u}_\omega,$ for any $g\in G$ and $\omega\in \Omega$. This implies that
$GK=K$.

Similarly, one can obtain $GK=K$ provided that (Alt$_b$) is satisfied.  Thus we have concluded the statement (i) of Theorem B.

Finally we assume that (Alt$_c$) holds. Then fix any $\omega\in
\Omega$, $G\bar{u}_{\omega}$ is a connected, locally compact and
totally ordered set in $X$. Moreover, $G\bar{u}_{\omega}$  has
neither the l.u.b. nor the g.l.b. in $X$. It then follows from
\cite[Appendix, Proposition Y2, Page 434]{OM1} that
$G\bar{u}_{\omega}$ coincides with the image of a strictly
order-preserving continuous path in $X$:
\begin{equation}\label{increase-path}
J_\omega: \RR\rightarrow G\bar{u}_{\omega}\subset X. \end{equation}
\vskip 1mm Motivated by \cite[Section 3]{CGW}, we choose an $\omega_0\in \Omega$ and define the mapping
\begin{equation}\label{defini-h}
H: \Omega\times \RR\to GK;\, (\omega,\alpha)\mapsto
\mathcal{O}(J_{\omega_0}(\alpha)) \cap P^{-1}(\omega),\end{equation}
where $J_{\omega_0}$ comes from \eqref{increase-path} with $\omega$
replaced by $\omega_0$. Then it is not hard to check (a)-(c) for $H$ in the statement
(ii) in Theorem B. We only need to show that $H$ is a bijective continuous map.

To end this, we first note that $H$ is surjective. Indeed, for any
$(g\bar{u}_{\omega},\omega)\in GK$, let the real number
$\hat{\alpha}\in \RR$ be such that
$J_{\omega_0}(\hat{\alpha})=g\bar{u}_{\omega_0}$. Then it is easy to
see that $\mathcal{O}(J_{\omega_0}(\hat{\alpha})) \cap
P^{-1}(\omega)=(g\bar{u}_{\omega},\omega),$ because $gK$ is a
uniformly stable $1$-cover of the base $\Omega$. Consequently,
$H(\omega,\hat{\alpha})=(g\bar{u}_{\omega},\omega)$, which implies
that $H$ is surjective.

Next we choose any
$(\omega_i,\alpha_i)\in \Omega\times \RR, i=1,2,$ with
$H(\omega_1,\alpha_1)=H(\omega_2,\alpha_2)$. For each $\alpha_i$, there is a $g_i\in G$
such that $J_{\omega_0}(\alpha_i)=g_i\bar{u}_{\omega_0}$ for $i=1,2$. Again by the $1$-cover property of $g_iK$,
$$(g_1\bar{u}_{\omega_1},\omega_1)=H(\omega_1,\alpha_1)=H(\omega_2,\alpha_2)=
(g_2\bar{u}_{\omega_2},\omega_2).$$  Combining with
\eqref{increase-path}, we obtain that $\omega_1=\omega_2$ and
$g_1=g_2$, which implies that $\alpha_1=\alpha_2$. Thus $H$ is
injective.

In order to prove $H$ is continuous, we choose any sequence
$\{(\omega_k,\alpha_k)\}_{k=1}^{\infty}\subset\Omega\times \RR$ with
$(\omega_k,\alpha_k)\to (\omega_\infty,\alpha_\infty)$ as $k\to
\infty$. Accordingly, for each $k=1,2,\cdots, \infty$, we can find
$g_k\in G$ such that $J_{\omega_0}(\alpha_k)=g_k\bar{u}_{\omega_0}$.
Similarly as above, one can further obtain that
\begin{equation}\label{H-expre}
H(\omega_k,\alpha_k)=(g_k\bar{u}_{\omega_k},\omega_k),\end{equation}
for $k=1,2,\cdots, \infty$. Since $\alpha_k\to \alpha_\infty$, we
have $g_k\bar{u}_{\omega_0}\to g_\infty\bar{u}_{\omega_0}$ as $k\to
\infty$. Note also that $g_\infty K$ is uniformly stable. Then for
any $\eps>0$, there exists an integer $N=N(\eps)>0$ such that
$\norm{u(t,g_k\bar{u}_{\omega_0},\omega_0)-u(t,g_\infty\bar{u}_{\omega_0},\omega_0)}\le
\eps/3$ for all $k\ge N$ and $t\ge 0$. By letting $t\to \infty$, it
yields that, if $k\ge N$ then
\begin{equation}\label{aa1}
\norm{g_k\bar{u}_{\omega}-g_\infty\bar{u}_{\omega}}\le \eps/3,
\end{equation}
uniformly for all $\omega\in \Omega.$
 Moreover, for such
$\varepsilon$ and $N$ (choose $N$ larger if necessary), it is easy
to see that
\begin{equation}\label{a1}
\norm{\omega_k-\omega_\infty}<\varepsilon/3 \quad \textnormal{ and }
\quad
\norm{g_\infty\bar{u}_{\omega_k}-g_\infty\bar{u}_{\omega_\infty}}<\varepsilon/3,
\end{equation}
for all $k\ge N$.
 By virtue
of \eqref{H-expre}-\eqref{a1}, we have
\begin{eqnarray*}
\norm{H(\omega_k,\alpha_k)-H(\omega_\infty,\alpha_\infty)}&=&
\norm{(g_k\bar{u}_{\omega_k},\omega_k)-(g_\infty\bar{u}_{\omega_\infty},\omega_\infty)}\\
&\le &
\norm{\omega_k-\omega_\infty}+\norm{g_k\bar{u}_{\omega_k}-g_\infty\bar{u}_{\omega_k}}+
\norm{g_\infty\bar{u}_{\omega_k}-g_\infty\bar{u}_{\omega_\infty}}\\
&<& \varepsilon/3+\varepsilon/3+\varepsilon/3=\varepsilon,
\end{eqnarray*}
for all $k\ge N$. We have proved that $H$ is continuous.  $\qquad\qquad\qquad\qquad\qquad\qquad\qquad\square$

\vskip 3mm
\begin{proof}[Proof of Theorem C]
Since $G$ is compact,  both (Alt$_a$) and (Alt$_b$) are satisfied. Then we directly deduce that
$GK=K$ from the proof above.
\end{proof}

\section{Uniformly stability of $K$ with asymptotic phase}
In this section, we will prove the asymptotic phase of the uniformly stable minimal set $K$, i.e., Theorem D in
Section 3. We first present the following useful lemma:

\begin{lemma}\label{lemma3.11} Assume all the hypotheses in Theorem B are satisfied. Assume also that
$GK\ne K$. Then there exists a $\delta_0>0$ such that, if
$(u,\omega)\in B_{\delta_0}(K)$ satisfies
 $\mathcal{O}(u,\omega)\preceq g_1K$
  for some $g_1\in G$, then
$\mathcal{O}(u,\omega)=g_2K$ for some $g_2\in G$. The same
conclusion also holds if $(u,\omega)\in B_{\delta_0}(K)$ satisfies
 $\mathcal{O}(u,\omega)\succeq g_1K$.
\end{lemma}

\begin{proof} Without loss of generality, we only prove the first statement.
Suppose that there exists a sequence
$\{(u_m,\omega_m)\}_{m=1}^{\infty}\subset X\times \Omega$ such that, for each $m\ge 1$,

(i) $(u_m,\omega_m)\in B_{1/m}(K)$;

(ii) $\mathcal{O}(u_m,\omega_m)\preceq g_{m}^1K,$  for some $g_{m}^1\in G$; and

(iii) $\mathcal{O}(u_m,\omega_m)\neq gK, ~\mbox{ for any }~g\in G.$

\noindent By virtue of Lemma \ref{Novo-Ob}, (iii) implies that
\begin{equation}\label{equ.3.22}gK\nsubseteq\mathcal{O}(u_m,\omega_m)
\quad \textnormal{ for all } m\ge 1 \textnormal{ and } g\in G.
\end{equation} Now
we claim that
\begin{equation}\label{close-to-K}
d(\mathcal{O}(u_m,\omega_m), K)\to 0,\qquad \mbox{ as }m\to \infty.
\end{equation}
In fact, since $K$ is uniformly stable, for any
$\varepsilon>0$ there exists a $\tilde{\delta}(\varepsilon)$ such
that, if
$\norm{(y,\omega)-(\bar{u}_\omega,\omega)}<\tilde{\delta}(\varepsilon)$
then $\norm{u(t,y,\omega)-u(t,\bar{u}_\omega,\omega)}<\varepsilon$
for all $t\geq0$. Then, for $(u_m,\omega_m)\in B_{1/m}(K)$ with
$m$ sufficiently large, one has
$\norm{(u_m,\omega_m)-(\bar{u}_{\omega_m},\omega_m)}<\frac{1}{m}<\tilde{\delta}(\varepsilon)$,
and hence,
$\norm{u(t,u_m,\omega_m)-u(t,\bar{u}_{\omega_m},\omega_m)}<\varepsilon$
for all $t\geq0$. By the minimality of
$\Omega$, it then follows that
$\norm{(y,\omega)-(\bar{u}_\omega,\omega)}\leq\varepsilon$ whenever
$(y,\omega)\in\mathcal{O}(u_m,\omega_m)$. Thus we
have proved the claim.

Now fix $m\in\mathbb{N}$. We define $A_m=\{g\in
G:\mathcal{O}(u_m,\omega_m)\preceq gK\}$. Clearly, $A_m$ is nonempty (because
$g_{m}^1\in A_m$ by (ii)) and closed in $G$. By virtue of
(\ref{equ.3.22}) and \eqref{close-to-K}, one obtains that $A_m=\{g\in G:\mathcal{O}(u_m,\omega_m)\prec gK\}$,
 and moreover, $\mathcal{O}(u_m,\omega_m)\subset B_\delta(K)$ as long as
 $m$ is
sufficiently large. Here the $\delta$ is adopted from condition (A4) in Section 3.

As a consequence, (A4) entails that
$A_m$ is also open for all $m$ sufficiently large. Since $G$ is
connected,  $A_m=G$ for all $m$ sufficiently large.
This then implies that
$$\mathcal{O}(u_m,\omega_m)\preceq gK,~\forall g\in G,$$
for all $m$ sufficiently large.  By letting
$m\rightarrow\infty$ in the above inequality, \eqref{close-to-K} yields that
$K\preceq gK,~\forall g\in G.$
Replacing $g$ with $g^{-1}$ and applying $g$ on both sides, we get
$g K\preceq K.$ Hence $gK=K$ for all $g\in G$, a contraction. We have completed
the proof of the lemma.
\end{proof}

\vskip 2mm
\begin{proof}[Proof of Theorem D] Let $\delta_0>0$ be defined in
Lemma \ref{lemma3.11}. We take a $\delta_*\in
(0,\min\{\delta,\delta_0\})$
 such that $(u\wedge\bar{u}_\omega,\omega)\in B_{\delta_0}(K)$ whenever $(u,\omega)\in B_{{\delta_*}}(K)$.
 Since $u\wedge\bar{u}_\omega\leq\bar{u}_\omega$, one has $\mathcal{O}(u\wedge\bar{u}_\omega,\omega)\preceq
K.$ It then follows from Lemma \ref{lemma3.11} that
$\mathcal{O}(u\wedge\bar{u}_\omega,\omega)=g_*K$ for some $g_*\in G$. Note also that
$u\wedge\bar{u}_\omega\leq u.$ Then $g_*K\preceq\mathcal{O}(u,\omega)$. Applying Lemma \ref{lemma3.11} again, we
obtain that $\mathcal{O}(u,\omega)=gK$ for some $g\in G$. This
completes the proof.
\end{proof}

\section{Proof of Lemma A}

\begin{proof}[Proof of Lemma A] First we
shall show that {\it there exists a neighborhood $B(e)\subset G$ of $e$
such that for any $g\in B(e)$, one has
$g\bar{u}_{\omega_0}\leq\bar{u}_{\omega_0}$ or
$g\bar{u}_{\omega_0}\geq\bar{u}_{\omega_0}$ for some
$\omega_0\in\Omega$}. Otherwise, one can find a sequence $\{g_n\}_{n=0}^{\infty}\subset G$ with
$g_n\rightarrow e$ as $n\rightarrow\infty$ such that
\begin{equation}\label{unorder-sequ}
g_n\bar{u}_\omega\nleq\bar{u}_\omega~\mbox{ and }~g_n\bar{u}_\omega\ngeq\bar{u}_\omega,~\mbox{ for all }n\geq0\mbox{ and }\omega\in\Omega.
\end{equation}

In what follows, we will deduce a contradiction from
\eqref{unorder-sequ}. For this purpose,  we fix an
$\omega_0\in\Omega$, and due to (A4-i), we define
$K_n=\mathcal{O}(g_n\bar{u}_{\omega_0}\wedge
\bar{u}_{\omega_0},\omega_0)$ for all $n$ sufficiently large.
Without loss of generality, one may also assume that $K_n$ is
defined for all $n\in\mathbb{N}$. Clearly,
$K=\mathcal{O}(\bar{u}_{\omega_0},\omega_0)$. Then one can obtain
the following three facts, the proof of which will be presented in
the end of this section (see Propositions
\ref{prop.3.3}-\ref{lemma3.4}):

{\bf (F1)} $K_n\prec K$ and $K_n\prec g_nK$ for all
$n\in\mathbb{N}$.

{\bf (F2)}~$d(K_n,K)\rightarrow 0$, as $n\rightarrow\infty$.

{\bf (F3)} Given the $\delta>0$ in (A4), there exists a neighborhood
$\hat{B}(e)\subset G$ of $e$ and $N_0\in\mathbb{N}$ such that
$$d(gK_n,K)\leq\delta~\mbox{ and }~d(g_n^{-1}gK_n,K)\leq\delta,$$
for all $g\in\hat{B}(e)$ and $n\geq N_0$.

\vskip 2mm
For such $\hat{B}(e)$ and $N_0\in \mathbb{N}$ in (F3), we take a
neighborhood $B(e)\subset G$ of $e$ with
$B(e)\subset\overline{B(e)}\subset\hat{B}(e)$, and
define
$$A_n=\{g\in\overline{B(e)}:gK_n\preceq K\mbox{ and }g_n^{-1}gK_n\preceq K\}$$ for each $n\ge N_0$.
By (F1), it is easy to see that
$e\in A_n\neq\emptyset$. Moreover, $A_n$ is
closed in $\overline{B(e)}$. We assert that
\begin{equation}\label{equ.3.6}A_n=\{g\in\overline{B(e)}:gK_n\prec K\mbox{ and
}g_n^{-1}gK_n\prec K\}.\end{equation}  Indeed, for $g\in A_n$,
suppose that there exists some $(y,\tilde{\omega})\in K_n$ such that
$gy=\bar{u}_{\tilde{\omega}}$. Then by $g_n^{-1}gK_n\preceq K$ we
have $g_n^{-1}gy\leq \bar{u}_{\tilde{\omega}}$. It entails that
$g_n^{-1}\bar{u}_{\tilde{\omega}}\leq \bar{u}_{\tilde{\omega}}$, and
hence $\bar{u}_{\tilde{\omega}}\leq g_n\bar{u}_{\tilde{\omega}}$,
contradicting to \eqref{unorder-sequ}. Similarly, for such $g\in
A_n$, suppose that there exists $(z,\hat{\omega})\in K_n$ such that
$g_n^{-1}gz=\bar{u}_{\hat{\omega}}$. Then by $gK_n\preceq K$ we have
$gz\leq \bar{u}_{\hat{\omega}}$, which yields
$g_n^{-1}\bar{u}_{\hat{\omega}}\geq
g_n^{-1}gz=\bar{u}_{\hat{\omega}}$, and hence
$\bar{u}_{\hat{\omega}}\geq g_n\bar{u}_{\hat{\omega}}$,
contradicting to \eqref{unorder-sequ} again. So we have proved the
assertion (\ref{equ.3.6}).

Now fix $n\ge N_0$ and let $g\in A_n$, we write
$v_{g,n}^0:=g(g_n\bar{u}_{\omega_0}\wedge\bar{u}_{\omega_0})$ and
$w_{g,n}^0:=g_n^{-1}g(g_n\bar{u}_{\omega_0}\wedge\bar{u}_{\omega_0})$. Then by (F3) and Lemma \ref{lemma3.6},
one obtains that
\begin{equation*}
d(\mathcal{O}(v_{g,n}^0,\omega_0), K)\leq
\delta~\mbox{ with }~
\mathcal{O}(v_{g,n}^0,\omega_0)=gK_n\prec K,
\end{equation*}
and
\begin{equation*}
d(\mathcal{O}(w_{g,n}^0,\omega_0), K)\leq
\delta~\mbox{ with }~
\mathcal{O}(w_{g,n}^0,\omega_0)=g_n^{-1}gK_n\prec K.
\end{equation*} Accordingly, the condition (A4) implies that
there exist neighborhoods $B_1(e)$, $B_2(e)\subset G$ of $e$ such
that
$gK_n=\mathcal{O}(v_{g,n}^0,\omega_0)\prec B_1(e)K$
and $g_n^{-1}gK_n=\mathcal{O}(w_{g,n}^0,\omega_0)\prec
B_2(e)K$, where $B_i(e)K=\{gK:g\in B_i(e)\}$ for $i=1,2$. As a consequence,
\begin{equation}\label{equ.3.12}
(B_1(e))^{-1}gK_n\prec K~~\mbox{ and }~~(B_2(e))^{-1}g_n^{-1}gK_n\prec K.
\end{equation}
Clearly, $(B_1(e))^{-1}g$ and $(B_2(e))^{-1}g_n^{-1}g$ are neighborhoods of $g$ and $g_n^{-1}g$, respectively.
Moreover, by the continuity of $g\mapsto g_n^{-1}g$,
one can find a neighborhood $V_g$ of $g$ in $G$, such that
$g_n^{-1}V_g\subset(B_2(e))^{-1}g_n^{-1}g$. Thus by (\ref{equ.3.12})
we have $g_n^{-1}V_gK_n\prec K$. Now let
$W_g:=\overline{B(e)}\cap V_g\cap(B_1(e))^{-1}g$.
Then by (\ref{equ.3.12}) again, $W_g$ is a neighborhood of $g$ in $\overline{B(e)}$ satisfying
$$W_gK_n\prec K~\mbox{ and }~g_n^{-1}W_gK_n\prec K.$$ Therefore, $W_g\subset A_n$, which implies that
 $A_n$ is also open in
$\overline{B(e)}$. Thus by the connectivity of
$G$ (and hence the connectivity of $\overline{B(e)}$), one has
$$A_n=\overline{B(e)},~\forall n\geq N_0.$$
Consequently,
$$\overline{B(e)}K_n\preceq K ~\mbox{ and }~g_n^{-1}\overline{B(e)}K_n\preceq K$$  for all $n\geq N_0$.
Letting $n\rightarrow\infty$ in the above, by (F2), we then have
\begin{equation}\label{equ.3.13}\overline{B(e)}K\preceq K.\end{equation} Since $g_n\rightarrow e$ as
$n\rightarrow\infty$, (\ref{equ.3.13}) implies that $g_n
\bar{u}_\omega\leq \bar{u}_\omega$ for all $\omega\in\Omega$ and $n$
sufficiently large, which is a contradiction to \eqref{unorder-sequ}.

Therefore, we have proved  that there exists a neighborhood $B(e)\subset G$
of $e$ such that for any $g\in B(e)$, one has
$g\bar{u}_{\omega_0}\leq\bar{u}_{\omega_0}$ or
$g\bar{u}_{\omega_0}\geq\bar{u}_{\omega_0}$ for some
$\omega_0\in\Omega$.

Without loss of generality, we assume that
$g\bar{u}_{\omega_0}\leq\bar{u}_{\omega_0}$. Then the monotonicity of $\Pi_t$
implies
$g\bar{u}_{\omega_0\cdot t}\leq\bar{u}_{\omega_0\cdot t}$ for any $t\ge 0$. Now for any $\omega\in\Omega$, we choose
a sequence $\{t_n\}\rightarrow\infty$  such that
$\omega_0\cdot t_n\rightarrow\omega$ as $n\rightarrow\infty$. By the 1-cover property of $K$, one has
$\bar{u}_{\omega_0\cdot t_n}\rightarrow\bar{u}_{\omega}$ as
$n\rightarrow\infty$.
Thus, by letting $n\rightarrow\infty$, we obtain that
$g\bar{u}_\omega\leq\bar{u}_\omega$ for any $\omega\in\Omega$. This implies that
$gK\preceq K$ for any $g\in B(e).$ Similarly, one can obtain that $K\preceq gK$ for any $g\in B(e)$ provided that $\bar{u}_{\omega_0}\leq g\bar{u}_{\omega_0}.$ Accordingly, we conclude that for
$K=\{(\bar{u}_\omega,\omega):\omega\in \Omega\}$, there holds
$$gK\preceq K\mbox{ or }gK\succeq K,~\forall g\in B(e).$$ This is the exact statement of Lemma A.
\end{proof}

Finally, it only left to
check (F1)-(F3) above. This will be done in the following three propositions.

\begin{proposition}\label{prop.3.3}
{\rm (F1)} holds, i.e., $K_n\prec K$ and $K_n\prec g_nK$ for all
$n\in\mathbb{N}$.
\end{proposition}
\begin{proof}
Note that $g_n\bar{u}_{\omega_0}\wedge
\bar{u}_{\omega_0}<\bar{u}_{\omega_0}$ (resp.
$<g_n\bar{u}_{\omega_0}$).  It then follows from the monotonicity of
$\Pi_t$ that
\begin{equation}\label{equ.3.2}\Pi_{t}(g_n\bar{u}_{\omega_0}\wedge
\bar{u}_{\omega_0},\omega_0)\leq\Pi_{t}(\bar{u}_{\omega_0},\omega_0)~\mbox{
(resp. }\le \Pi_{t}(g_n\bar{u}_{\omega_0},\omega_0)),\end{equation} for all $t\ge 0$. So,
for any $(x,\omega)\in K_n$, one can find a sequence
$\{t_k\}\rightarrow\infty$ ($k\rightarrow\infty$) such that
$\Pi_{t_k}(g_n\bar{u}_{\omega_0}\wedge
\bar{u}_{\omega_0},\omega_0)\rightarrow(x,\omega)$ as
$k\rightarrow\infty$. Since $K$ is a $1$-cover, one has
$\Pi_{t_k}(\bar{u}_{\omega_0},\omega_0)\to (\bar{u}_{\omega},\omega).$ Then
\eqref{equ.3.2} implies that $(x,\omega)\le (\bar{u}_{\omega},\omega)$. As a consequence,
$K_n\preceq K$. Similarly, we can also obtain $K_n\preceq g_nK$ for every $n\in \mathbb{N}$.

Now we claim that $K_n\prec K$ (resp.  $\prec g_nK$) for all $n\in
\mathbb{N}$. Otherwise, there
exist some $N\in \mathbb{N}$ and $(x,\tilde{\omega})\in K_N$
such that
\begin{equation}\label{equ.3.3}(x,\tilde{\omega})=(\bar{u}_{\tilde{\omega}},\tilde{\omega})
~~(\mbox{resp.}~
(=g_N\bar{u}_{\tilde{\omega}},\tilde{\omega})).\end{equation} Choose
a sequence $\{s_k\}\rightarrow\infty$ ($k\rightarrow\infty$) such
that $\Pi_{s_k}(g_N\bar{u}_{\omega_0}\wedge
\bar{u}_{\omega_0},\omega_0)\rightarrow(x,\tilde{\omega})$ as
$k\rightarrow\infty$. Since
\begin{equation*}
\begin{aligned} \Pi_{t}(g_n\bar{u}_{\omega}\wedge
\bar{u}_{\omega},\omega)&\leq\Pi_{t}(g_n\bar{u}_{\omega},\omega)\wedge
\Pi_{t}( \bar{u}_{\omega},\omega) \\&=(g_n\bar{u}_{\omega\cdot
t},\omega\cdot t)\wedge (\bar{u}_{\omega\cdot t},\omega\cdot t)=
(g_n\bar{u}_{\omega\cdot t}\wedge \bar{u}_{\omega\cdot
t},\omega\cdot t)\end{aligned}
\end{equation*} for all $\omega\in\Omega$, $t\geq0$ and $n\in\mathbb{N}$,  it follows
that
$$\Pi_{s_k}(g_N\bar{u}_{\omega_0}\wedge
\bar{u}_{\omega_0},\omega_0)\leq(g_N\bar{u}_{\omega_0\cdot
s_k}\wedge \bar{u}_{\omega_0\cdot s_k},\omega_0\cdot s_k).$$ Letting
$k\rightarrow\infty$ in the above,  by the continuity of
$\bar{u}_\omega$ w.r.t. $\omega\in\Omega$, we then get
$$(x,\tilde{\omega})\leq (g_N\bar{u}_{\tilde{\omega}}\wedge \bar{u}_{\tilde{\omega}},\tilde{\omega})<(\bar{u}_{\tilde{\omega}},\tilde{\omega})
~(\mbox{resp. }(<g_N\bar{u}_{\tilde{\omega}},\tilde{\omega})),$$ where the last inequality
is from \eqref{unorder-sequ}. Accordingly, a
contradiction to (\ref{equ.3.3}) is obtained. Thus we have proved $K_n\prec K$ (resp.  $\prec g_nK$) for all $n\in
\mathbb{N}$.
\end{proof}

\begin{proposition}\label{prop.3.2}
{\rm (F2)} holds, i.e., $d(K_n,K)\rightarrow 0$, as $n\rightarrow\infty$.
\end{proposition}
\begin{proof}
Note that $g_n\bar{u}_{\omega_0}\wedge\bar{u}_{\omega_0}\rightarrow\bar{u}_{\omega_0}$
 as $n\rightarrow\infty.$ Since $K$ is a uniformly
stable 1-cover of $\Omega$, it entails that, for any $\varepsilon>0$, there is some
$N_1\in \mathbb{N}$ such that \begin{equation}\label{equ.3.1}\norm{u(t,g_n\bar{u}_{\omega_0}
\wedge\bar{u}_{\omega_0},\omega_0)-\bar{u}_{\omega_0\cdot t}}<\eps
\end{equation} for all $n\ge N_1$ and $t\ge 0$. Choose any $(x,\omega)\in K_n$, there exists a sequence
$\{t_k\}\rightarrow\infty$ ($k\rightarrow\infty$) such that
$\Pi_{t_k}(g_n\bar{u}_{\omega_0}\wedge
\bar{u}_{\omega_0},\omega_0)\rightarrow(x,\omega)$ as
$k\rightarrow\infty$. By taking a subsequence, if necessary, we get
that
$\Pi_{t_k}(\bar{u}_{\omega_0},\omega_0)\rightarrow(\bar{u}_{\omega},\omega)$
as $k\rightarrow\infty$. Hence by (\ref{equ.3.1}), we have that
$\norm{x-\bar{u}_{\omega}}\leq\varepsilon$ for all $(x,\omega)\in
K_n$ and $n\geq N_1.$ Recall that $d(K_n,K)=\sup_{(x,\omega)\in
K_n}\norm{x-\bar{u}_\omega}$. Consequently, $d(K_n,K)\le \eps$ for
all $n\ge N_1$, which implies that $d(K_n,K)\to 0$ as $n\to \infty$.
\end{proof}

\begin{proposition}\label{lemma3.4}
{\rm (F3)} holds, i.e., for the $\delta>0$ in ${\rm (A4)}$, there
exists a neighborhood $\hat{B}(e)\subset G$ of $e$ and
$N_0\in\mathbb{N}$ such that $$d(gK_n,K)\leq\delta~\mbox{ and
}~d(g_n^{-1}gK_n,K)\leq\delta,$$ for all $g\in\hat{B}(e)$ and $n\geq
N_0$.
\end{proposition}
\begin{proof}
Firstly, suppose that there exist a sequence
$\{\tilde{g}_n\}_{n=0}^{\infty}\subset G$ with
$\tilde{g}_n\rightarrow e$ and a subsequence of $\{K_n\}_{n=0}^{\infty}$,
still denoted by $\{K_n\}_{n=0}^{\infty}$, such that
 $$d(\tilde{g}_nK_n,K)=\sup_{(y,\omega)\in
K_n}\norm{\tilde{g}_ny-\bar{u}_{\omega}}>\delta$$ for all $n\in \mathbb{N}$. Then one can
choose some $(y_n,\omega_n)\in K_n$ such that
\begin{equation}\label{equ.3.4}\norm{\tilde{g}_ny_n-\bar{u}_{\omega_n}}>\delta.\end{equation}
Without loss of generality we assume that
$\omega_n\rightarrow\omega$ in $\Omega$ as $n\rightarrow\infty$. Now
we claim that $y_n\rightarrow\bar{u}_\omega$ as
$n\rightarrow\infty$. Indeed, Proposition \ref{prop.3.2} suggests that,
for any $\varepsilon>0$,
there exists a positive integer $N\in\mathbb{N}$ such that
$$\norm{z-\bar{u}_{\omega}}<\varepsilon,~~\mbox{ for all }(z,\omega)\in
K_n~\mbox{ and }n> N.$$ So $\norm{y_n-\bar{u}_{\omega_n}}<\varepsilon\mbox{ for all }n>N,$ because
 $(y_n,\omega_n)\in K_n$.
Due to the continuity of $\bar{u}_\omega$ w.r.t.
$\omega\in\Omega$, one has
$$\norm{y_n-\bar{u}_{\omega}}\leq\norm{y_n-\bar{u}_{\omega_n}}+\norm{\bar{u}_{\omega_n}-\bar{u}_{\omega}}
<\varepsilon+\varepsilon=2\varepsilon,~~\forall n>\bar{N}$$ for some
positive integer $\bar{N}>N$. Thus, we have proved the claim. Then
by letting $n\rightarrow\infty$ in (\ref{equ.3.4}), we obtain
$\norm{\bar{u}_\omega-\bar{u}_\omega}=\norm{e\bar{u}_\omega-\bar{u}_\omega}\geq\delta,$
a contradiction. Such contradiction implies that one can find  a
neighborhood $B_1(e)$ of $e$ and some $N_1\in \mathbb{N}$ such that
$d(gK_n,K)\leq\delta$ for all $g\in B_1(e)$ and $n\ge N_1$.

Secondly, suppose that there exist a sequence
$\{h_n\}_{n=0}^{\infty}\subset G$ with $h_n\rightarrow e$ and a subsequence $\{K_{j_n}\}_{n=0}^{\infty}$ of
$\{K_n\}_{n=0}^{\infty}$
such that
$$d(g_{j_n}^{-1}h_nK_{j_n}, K)=\sup_{(y,\omega)\in
K_{j_{_n}}}\norm{g_{j_n}^{-1}h_ny-\bar{u}_{\omega}}>\delta~~\mbox{ for all
}n\in \mathbb{N}.$$ Then there exists some $(y_{j_n},\omega_{j_{_n}})\in
K_{j_n}$ such that
$\norm{g_{j_n}^{-1}h_ny_{j_n}-\bar{u}_{\omega_{j_{_n}}}}>\delta.$
Noticing $g_{j_n}^{-1}\to e$, one can
repeat the same argument above to deduce a contradiction. Thus, again one can find a neighborhood $B_2(e)$ of $e$ and some
$N_2\in \mathbb{N}$ such that $d(g_n^{-1}gK_n,K)\leq\delta$ for all $g\in B_2(e)$ and $n\ge N_2$.

Finally, let $\hat{B}(e)=B_1(e)\cap B_2(e)$ and $N_0=\max\{N_1,N_2\}$. We have completed the proof of (F3).
\end{proof}

\section{Applications to parabolic equations}

In this section we give some examples of second order parabolic
equations in time-recurrent structures which generate monotone
skew-product semiflows satisfying (A1)-(A4).

\subsection{Rotational symmetry}

Assume that $\Omega\subset\mathbb{R}^n$ is a (possibly unbounded)
rotationally symmetric domain with smooth boundary $\partial
\Omega$. Let $G$ be a connected closed subgroup of the rotation
group $SO(n)$. $\Omega$ is called $G$-symmetric if it is
$G$-invariant in the sense that $gx\in\Omega$ whenever $x\in\Omega$
and $g\in G$. A typical example of such a bounded domain is a ball,
a spherical shell, a solid torus or any other body of rotation.
While, typical unbounded domains include cylindrical domain or
$\RR^n$ itself. In \cite{WY}, asymptotic symmetry has been investigated for the bounded domains. In
this section, we focus on unbounded domains and, for brevity, we
will present the following example on $\RR^n$. As a matter of fact, general
unbounded $G$-symmetric  domains can be dealt with as well.

\vskip 3mm
Consider the following initial value problem on $\RR^n$:
\begin{equation}\label{IVP-sys}
\left\{\begin{split} &\dfrac{\partial u}{\partial t}= \Delta u
+f(t,x,u), \quad x\in \R^n,
\, t>0,\\
&u(0,x)=u_{0}(x), \quad\quad\quad x\in \R^n.
 \end{split}
 \right.
\end{equation}
Here the nonlinearity $f: \mathbb{R}\times \mathbb{R}^n\times
\mathbb{R}\to \mathbb{R}$ is assumed to be a $C^1$-admissible (with
$D=\mathbb{R}^{n+1}$) and uniformly almost periodic in $t$,
real-valued function.

In what follows we assume that

\vskip 2mm {\bf (f 1)} $f(t,gx,u)=f(t,x,u)$ for all
$x\in\mathbb{R}^n,\,u\in\mathbb{R},\,g\in G$ and $t\in\mathbb{R}$;

\vskip 2mm {\bf (f 2)} $f(t,x,0)=0$ for all $x\in \RR^n$ and $t\in\mathbb{R}$;

\vskip 2mm {\bf (f 3)} there exist positive numbers $\epsilon_0,
R_0,\alpha$ such that $\frac{\partial f}{\partial u}(t,x,u)\leq
-\alpha$ for all $|x|\geq R_0$, $|u|\leq\epsilon_0$ and
$t\in\mathbb{R}$.

\vskip 2mm Let $X$ be defined by
$$C_{{\rm unif}}(\RR^n)=\{u(x):u~ \mbox{ is bounded and uniformly continuous on }\RR^n\}$$
with the $L^\infty$-topology.
Let $Y=H(f)$ be the hull of the nonlinearity
$f$. Then, for any $g\in Y$, the function $g$ is uniformly almost
periodic in $t$ and satisfies all the above assumptions (f 1)-(f 3).
As a consequence, \eqref{IVP-sys} gives rise to a family of
equations associated to each $g\in Y$:
\begin{equation}\tag{\ref{IVP-sys}$_g$}
\left\{\begin{split} &\dfrac{\partial u}{\partial t}= \Delta u
+g(t,x,u), \quad x\in \R^n,
\, t>0,\\
&u(0,x)=u_{0}(x), \quad\quad\quad x\in \R^n.
 \end{split}
 \right.
\end{equation}
By standard theory for parabolic equations (see \cite{Fri,Hen}),
for every $u_0\in X$ and $g\in H(f)$, equation $(\ref{IVP-sys}_g)$ admits a
(locally) unique classical solution $u(t,\cdot,u_0,g)$ in $X$ with
$u(0,\cdot,u_0,g)=u_0$. This solution also continuously depends on
$g\in Y$ and $u_0\in X$ (see, e.g. \cite{Hen, MS}). Therefore, $(\ref{IVP-sys}_g)$ defines a
(local) skew-product semiflow $\Pi_t$ on $X\times Y$ with
$$\Pi_t(u_0,g)=(u(t,\cdot,u_0,g),g\cdot t),\quad \forall\,
(u_0,g)\in X\times Y, \,t\ge 0.$$ We define an order relation in $X$
by $$u\leq v~\mbox{ if }~u(x)\leq v(x),~\forall x\in\RR^n.$$ The
action of $G$ on $\RR^n$ induces a group action on $X$ by
$$a:u(x)\mapsto u(a^{-1}x).$$ Clearly, (A1)-(A3) in Section 3 are fulfilled.

\begin{theorem}[Rotational symmetry]\label{theorem 4.6} Any
uniformly $L^\infty$-stable entire (possibly sign-changing) solution $\bar{u}_f(t,x)$ of
{\rm (\ref{IVP-sys})} (with
$\mathcal{M}(\bar{u}_f)\subset\mathcal{M}(f)$) satisfying
\begin{equation}\label{equ.4.21} \sup_{t\in\RR}|\bar{u}_f(t,x)|\to
0,\quad\mbox{ as
}|x|\to\infty
\end{equation} is G-symmetric, i.e., $\bar{u}_f(t,gx)=\bar{u}_f(t,x)$
for all $t\in\RR$, $x\in\RR^n$ and $g\in G$.
\end{theorem}

For the entire solution $\bar{u}_f(t,x)$ given in Theorem
\ref{theorem 4.6}, clearly, $E:={\rm cl}\{\bar{u}_f(t,\cdot)\in
X:t\in\RR\}$ is a 1-cover of $H(f)$, because $\bar{u}_f$ is uniformly stable.
 Thus one can write
$E=\{\bar{u}_g(0,\cdot)\in X:g\in H(f)\}$ with
$\bar{u}_f(t,\cdot)=\bar{u}_{f\cdot t}(0,\cdot)$ for all $t\in\RR$.
Let $K:=\{(\bar{u}_g(0,\cdot),g):g\in H(f)\}$.

Recall that the rotation group $G$ is compact, in order to obtain
the rotational symmetry of $\bar{u}_f(t,x)$, we only need to check
(A4) in view of our abstract Theorem~C. This will be done in
Propositions \ref{lemma 4.8} and \ref{compactness-or} below. We
first proceed to present the following useful lemma.

\begin{lemma}\label{lemma 4.7}
%($\ref{equ.4.21}_g$) holds for any $g\in H(f)$, that
%is, for any $g\in H(f)$ we have $$\sup_{t\in\RR}|\bar{u}_g(t,x)|\to
%0\mbox{ and }\sup_{t\in\RR}|\nabla\bar{u}_g(t,x)|\to 0\quad\mbox{ as
%}|x|\to\infty.$$ Furthermore,
$$\sup_{g\in H(f)}\sup_{t\in\RR}|\bar{u}_g(t,x)|\to 0,
%\mbox{ and
%}\sup_{g\in H(f)}\sup_{t\in\RR}|\nabla\bar{u}_g(t,x)|\to 0
\quad\mbox{ as }|x|\to\infty.$$
%Consequently, for the $\epsilon_0$,
%$R_0$ in {\rm (f 3)} there exists a $R\geq R_0>0$ such that
%$$|\bar{u}_g(t,x)|<\epsilon_0\mbox{ and }
%|\nabla\bar{u}_g(t,x)|<\epsilon_0\quad\mbox{ for all
%}x\in\Omega\backslash\Omega_R,~g\in H(f)\mbox{ and }t\in\R,$$ where
%$\Omega_R=\{x\in\Omega:|x|<R\}$.
\end{lemma}
\begin{proof} Since $K$ is a 1-cover of $H(f)$, for any $g\in H(f)$
there exists a sequence $\{t_n\}\to \infty$ such
that
$$\lim_{n\to\infty}\bar{u}_{f\cdot t_n}(t,x)=\lim_{n\to\infty}\bar{u}_{(f\cdot t_n)\cdot t}(0,x)=\bar{u}_{g\cdot t}(0,x)=\bar{u}_g(t,x)$$
uniformly in $t\in\RR$ and $x\in\RR^n$. Then for any
$\varepsilon>0$, it follows from (\ref{equ.4.21}) that there exists
some $R_\varepsilon>0$ such that
\begin{equation*}\begin{aligned}|\bar{u}_g(t,x)|&\leq|\bar{u}_g(t,x)-\bar{u}_{f\cdot
t_n}(t,x)|+|\bar{u}_{f\cdot
t_n}(t,x)|\\&=|\bar{u}_g(t,x)-\bar{u}_{f\cdot
t_n}(t,x)|+|\bar{u}_{f}(t+t_n,x)|\\&<\frac{\varepsilon}{2}+\frac{\varepsilon}{2}=\varepsilon
\end{aligned}
\end{equation*} for all $t\in\RR$, $|x|>R_\varepsilon$, $g\in H(f)$ and $n$ sufficiently large. This
implies that $$\sup_{g\in H(f)}\sup_{t\in\RR}|\bar{u}_g(t,x)|\to
0\quad\mbox{ as }|x|\to\infty.$$
\end{proof}

\begin{proposition}\label{lemma 4.8}
Let $\epsilon_0$ be given in {\rm (f 3)}. Let also $(u_0,g_0)\in X\times H(f)$ be such that its omega limit
set $\mathcal{O}(u_0,g_0)$ exists and satisfies
$$
\norm{v(\cdot)-\bar{u}_g(0,\cdot)}_{L^\infty}<\frac{\epsilon_0}{2},\quad\mbox{for
all }(v,g)\in\mathcal{O}(u_0,g_0), $$ with
$$(v(x),g)\leq (\bar{u}_g(0,x),g),~v(x)\not\equiv
\bar{u}_g(0,x),\quad\quad x\in\RR^n,\,(v,g)\in\mathcal{O}(u_0,g_0).$$
Then there is a neighborhood $B(e)\subset G$ of $e$ such that
$$(av(x),g)\leq (\bar{u}_g(0,x),g),~av(x)\not\equiv
\bar{u}_g(0,x),$$ for all $x\in\RR^n$, $a\in B(e)$ and $(v,g)\in\mathcal{O}(u_0,g_0)$.
 The assertion remains true if the inequality sign
$\leq$ is replaced by $\geq$.
\end{proposition}

\begin{proof}  We only prove the first assertion of the Proposition. The last assertion is similar.
Motivated by \cite[Lemma 5.8]{OM1,OM2}, we let
$\alpha,\,\epsilon_0,\,R_0$
 be such that (f 3) holds. By virtue of Lemma \ref{lemma 4.7},
 we choose some $R\geq R_0>0$ such that
\begin{equation}\tag{{\bf A}}|\bar{u}_g(t,x)|<\frac{\epsilon_0}{4},\quad\mbox{ for all
}x\in\RR^n\backslash B_R,~g\in H(f)\mbox{ and }t\in\RR,
\end{equation} where $B_R=\{x\in\RR^n:|x|<R\}$. Moreover,
for such $\epsilon_0>0$, there exists a neighborhood $B_0(e)\subset G$ of $e$
such that
\begin{equation}\label{equ.4.27}|a\bar{u}_g(0,x)-\bar{u}_g(0,x)|<\frac{\epsilon_0}{4}, ~\mbox{
for all } x\in\RR^n, a\in B_0(e) \mbox{ and }g\in H(f).
\end{equation} Recall that
\begin{equation}\label{equ.4.25}
\norm{v(\cdot)-\bar{u}_g(0,\cdot)}_{L^\infty}<\frac{\epsilon_0}{2},\quad\mbox{for
all }(v,g)\in\mathcal{O}(u_0,g_0).\end{equation}
It then follows from (\ref{equ.4.27})-(\ref{equ.4.25}) and (A) that
\begin{eqnarray*}
|av(x)|&\le&|av(x)-a\bar{u}_g(0,x)|+|a\bar{u}_g(0,x)|
\\&\le&|v(a^{-1}x)-\bar{u}_g(0,a^{-1}x)|+|a\bar{u}_g(0,x)-\bar{u}_g(0,x)|+
|\bar{u}_g(0,x)|\\&<&\frac{\epsilon_0}{2}+\frac{\epsilon_0}{4}+\frac{\epsilon_0}{4}=\epsilon_0,
\end{eqnarray*} for all $a\in B_0(e)$, $x\in\RR^n\backslash B_R$ and $(v,g)\in\mathcal{O}(u_0,g_0)$. That is,
\begin{equation}\tag{{\bf B}} |av(x)|<\epsilon_0~\mbox{ for all }
a\in B_0(e),\,x\in\RR^n\backslash B_R\mbox{ and
}(v,g)\in\mathcal{O}(u_0,g_0).
\end{equation}
%On the other hand, it is easy to see that for any $\epsilon>0$
%there exists a neighborhood $B_\epsilon(e)\subset G$ of $e$ such
%that, if $a\in B_\epsilon(e)$ then
%\begin{equation}\label{equ.4.26}|a\bar{u}_g(0,x)-\bar{u}_g(0,x)|<\epsilon
%\end{equation} for all $x\in\RR^n$ and $g\in H(f)$.

Noticing that $(v(x),g)\leq
(\bar{u}_g(0,x),g),~v(x)\not\equiv \bar{u}_g(0,x)$ for $x\in\RR^n$
and $(v,g)\in\mathcal{O}(u_0,g_0),$  the
strong maximum principle yields that
$$(u(t,x,v,g),g\cdot t)<(\bar{u}_{g\cdot t}(0,x),g\cdot t),\quad\forall  x\in\RR^n,\,(v,g)\in\mathcal{O}(u_0,g_0),\,t>0.$$
So, by the invariance of $\mathcal{O}(u_0,g_0)$, we obtain that
$(v(x),g)< (\bar{u}_g(0,x),g)$, for $x\in\RR^n$ and
$(v,g)\in\mathcal{O}(u_0,g_0)$. Since $\mathcal{O}(u_0,g_0)$ is
compact in $X\times H(f)$, the continuity of $\bar{u}_g(0,\cdot)$ on
$g$ implies that there is an $\tilde{\epsilon}>0$ such that
$$(v(x),g)<(\bar{u}_g(0,x)-\tilde{\epsilon},g),\quad\mbox{for all }x\in\overline{B_R}\mbox{ and }(v,g)\in\mathcal{O}(u_0,g_0).$$
As a consequence, there exists a smaller neighborhood $B(e)\subset
B_0(e)$ of $e$ such that
\begin{equation}\tag{{\bf C}}(av(x),g)<(\bar{u}_g(0,x),g)\quad\mbox{for all
}\,a\in B(e),\,x\in\overline{B_R}~\mbox{ and
}(v,g)\in\mathcal{O}(u_0,g_0).\end{equation}  Note also that
$\bar{u}_g(0,\cdot)-av(\cdot)\geq
\bar{u}_g(0,\cdot)-a\bar{u}_g(0,\cdot)$, for all
$(v,g)\in\mathcal{O}(u_0,g_0)$. Then one further obtains that
\begin{equation}\tag{{\bf D}}
\liminf\limits_{|x|\rightarrow\infty}(\bar{u}_g(0,x)-av(x))
\geq\liminf\limits_{|x|\rightarrow\infty}(\bar{u}_g(0,x)-a\bar{u}_g(0,x))=0
\end{equation} for all $(v,g)\in\mathcal{O}(u_0,g_0)$ and $a\in B(e)$.

Now we claim that the Proposition follows immediately from (A)-(D).
Indeed, for any $(v,g)\in\mathcal{O}(u_0,g_0)$ and $\tau>0$, one can
find some $(v_{-\tau},g_{-\tau})\in\mathcal{O}(u_0,g_0)$ such that
$\Pi_\tau(v_{-\tau},g_{-\tau})=(v,g).$ Then for any $a\in B(e)$, by
(A)-(D) and the invariance of $\mathcal{O}(u_0,g_0)$, we have that
\begin{equation*}\begin{aligned}{\rm (i)}\quad &|\bar{u}_g(t,x)|<\epsilon_0,~\mbox{ for all
}x\in\RR^n\backslash B_R, \,g\in H(f)\mbox{ and }t\in\R,\\
{\rm (ii)}\quad
&|au(t,x,v_{-\tau},g_{-\tau})|<\epsilon_0,\quad\mbox{ for all
}t>0\mbox{ and
}x\in\RR^n\backslash B_R,\\
{\rm (iii)}\quad
&au(t,x,v_{-\tau},g_{-\tau})<\bar{u}_{g_{-\tau}\cdot
t}(0,x),\quad\mbox{ for all }t>0\mbox{ and
}x\in\partial B_R,\\
{\rm (iv)}\quad
&\liminf\limits_{|x|\rightarrow\infty}(\bar{u}_{g_{-\tau}\cdot
t}(0,x)-au(t,x,v_{-\tau},g_{-\tau}))\geq 0, \quad\mbox{ for all
}t>0.\end{aligned}
\end{equation*} Therefore, Lemma \ref{lemma 4.9} below implies that
\begin{equation*}\bar{u}_{g_{-\tau}\cdot t}(0,x)-au(t,x,v_{-\tau},g_{-\tau})=\bar{u}_{g_{-\tau}\cdot t}(0,x)-u(t,x,av_{-\tau},g_{-\tau})\geq-2\epsilon_0 e^{-\alpha t}
\end{equation*} for all $x\in\RR^n\backslash B_R$ and
$t>0$. In particular (let $t=\tau$),
\begin{equation*}\bar{u}_{g_{-\tau}\cdot \tau}(0,x)-au(\tau,x,v_{-\tau},g_{-\tau})\geq-2\epsilon_0 e^{-\alpha \tau},\quad\mbox{for all }x\in\RR^n\backslash B_R,
\end{equation*} and hence
\begin{equation*}\bar{u}_g(0,x)-av(x)\geq-2\epsilon_0 e^{-\alpha \tau},\quad\mbox{for all
}x\in\RR^n\backslash B_R.
\end{equation*} Since $\tau>0$ is arbitrarily chosen, by letting
$\tau\rightarrow\infty$ we have $\bar{u}_g(0,x)\geq av(x)$, for all
$x\in\RR^n\backslash B_R$, $(v,g)\in\mathcal{O}(u_0,g_0)$ and $a\in
B(e)$. Combining with (C), we have completed the proof.
\end{proof}

\begin{lemma}\label{lemma 4.9} Let $\alpha,\,\epsilon_0,\,R_0$ be such that {\rm (f 3)} holds. Let $R\geq R_0$ be such that
$$|\bar{u}_g(t,x)|<\epsilon_0, \quad\mbox{ for all
}x\in\RR^n\backslash B_R,~g\in H(f)\mbox{ and }t\in\R.$$ Let also
$u(t,x,v_0,g)$ be a solution of ($\ref{IVP-sys}_g$) satisfying
\begin{equation*}|u(t,x,v_0,g)|<\epsilon_0,\quad\forall
t>0,~x\in\RR^n\backslash B_R.
\end{equation*}  Assume that
$$\bar{u}_g(t,x)\geq u(t,x,v_0,g),\quad\mbox{for }x\in\partial B_R,\,t>0$$
and
$$\liminf\limits_{|x|\rightarrow\infty}(\bar{u}_g(t,x)- u(t,x,v_0,g))\geq0,\quad\forall t>0.$$
Then
$$\bar{u}_g(t,x)- u(t,x,v_0,g)\geq-2\epsilon_0 e^{-\alpha t}\quad\mbox{for all }x\in\RR^n\backslash B_R\mbox{ and }t>0.$$
\end{lemma}

\begin{proof} The proof is similar as \cite[Lemma 5.9]{OM1}, we here give the detail for completeness. For any $g\in H(f)$, the function $w(t,x)=\bar{u}_g(t,x)-
u(t,x,v_0,g)$ is a solution of the linear parabolic equation
\begin{equation}\label{equ.4.22} \dfrac{\partial w}{\partial t}= \Delta
w +\xi(t,x)w,
\quad x\in \RR^n\backslash\overline{B_R}, \, t>0
\end{equation} under the boundary condition $w=\bar{u}_g- u\ge 0 $ on $\partial B_R$,
where
$$\xi(t,x)=\int_0^1g'_u(t,x,\theta\bar{u}_g(t,x)+(1-\theta)u(t,x,v_0,g))d\theta.$$  In view of our assumptions, it is easy to see that
$$|\theta\bar{u}_g(t,x)+(1-\theta)u(t,x,v_0,g)|<\epsilon_0\quad\mbox{for all }x\in\RR^n\backslash B_R\mbox{ and }t>0.$$
Since $g\in H(f)$ satisfies (f 3), we have
\begin{equation*} \xi(t,x)\leq -\alpha\quad\mbox{for all }x\in\RR^n\backslash B_R\mbox{ and }t>0.
\end{equation*} Let $\tilde{r}(t)=-2\epsilon_0 e^{-\alpha t}$.
Then $$\frac{\partial \tilde{r}}{\partial
t}\le \Delta
\tilde{r}+\xi(t,x)\tilde{r},\quad x\in
\RR^n\backslash\overline{B_R}, \, t>0.$$ Clearly, $\tilde{r}(t)< 0\le w(t,x)$ on $\partial B_R$.  Moreover,
\begin{equation*} \tilde{r}(0)=-2\epsilon_0\leq
\bar{u}_g(0,x)-v_0(x)=w(0,x),\quad\mbox{for
}x\in\RR^n\backslash B_R,
\end{equation*} and
$\tilde{r}(t)<0\leq\liminf\limits_{|x|\rightarrow\infty}w(t,x)$ for all $t>0.$
Then it follows from the comparison theorem
that $$\tilde{r}(t)\leq w(t,x)\quad\mbox{for all
}x\in\RR^n\backslash B_R\mbox{ and }t>0,$$ which
completes the proof.
\end{proof}

\begin{proposition}\label{compactness-or}
Let $\epsilon_0$ be given in {\rm (f 3)}.  Then, for any solution
$u(t,x,v_0,g)$ of {\rm ($\ref{IVP-sys}_g$)} satisfying
\begin{equation}\label{2}
\sup_{t\ge
0}\norm{u(t,\cdot,v_0,g)-\bar{u}_g(t,\cdot)}_{L^\infty}<\dfrac{\epsilon_0}{4},
\end{equation} the
forward orbit $O^+(v_0,g)$ is relatively compact in $X$.
\end{proposition}
\begin{proof}
Since
\begin{equation}\label{1}\sup_{t\in \RR}|\bar{u}_g(t,x)|\to 0~\mbox{
as }~|x|\rightarrow+\infty,\end{equation} let $R>R_0$ be such that
$\sup\limits_{t\in \RR}|\bar{u}_g(t,x)|\leq \epsilon_*$ for
$x\in\R^n\backslash B_{R}$, where $B_{R}=\{x\in \R^n:|x|<R\}$ and
$\epsilon_*=\frac{\epsilon_0}{4}$. In view of (\ref{2}), it yields
that \begin{equation}\label{3} |u(t,x,v_0,g)|\leq 2\epsilon_*\mbox{
for all }t\geq0 \mbox{ and } x\in \R^n\backslash B_{R}.
\end{equation}
Furthermore, $u(t,x,v_0,g)$ satisfies the initial boundary value
problem
\begin{equation}\label{IBVP-sys}
\left\{\begin{split} &\dfrac{\partial w}{\partial t}= \Delta w
+g(t,x,w), \quad x\in \R^n\backslash \overline{B_{R}},
\, t>0,\\
&w=\,u,\quad\quad \quad\quad x\in \p B_{R}, \, t>0,
\\
&w(0,x)=v_{0}(x), \quad\quad\quad x\in \R^n\backslash B_{R}.
 \end{split}
 \right.
\end{equation}
Now let $\phi^+$ satisfies
\begin{equation*}
\left\{\begin{split} &\dfrac{\partial \phi^+}{\partial t}= \Delta
\phi^+ -\alpha\phi^+, \quad x\in \R^n\backslash \overline{B_{R}},
\, t>0,\\
&\phi^+=\,3\epsilon_*,\quad\quad \quad\quad x\in \p B_{R}, \, t>0,
\\
&\phi^+(0,x)=3\epsilon_*, \quad\quad\quad x\in \R^n\backslash B_{R}.
 \end{split}
 \right.
\end{equation*}
Then $\hat{u}:=\bar{u}_g+\phi^+$ satisfies
\begin{equation*}\label{5}
\left\{\begin{split} &\dfrac{\partial u}{\partial t}= \Delta u
+g(t,x,\bar{u}_g)-\alpha\phi^+, \quad x\in \R^n\backslash
\overline{B_{R}},
\, t>0,\\
&u=\,3\epsilon_*+\bar{u}_g,\quad\quad \quad\quad x\in \p B_{R}, \,
t>0,
\\
&u(0,x)=3\epsilon_*+\bar{u}_g(0,x), \quad\quad\quad x\in
\R^n\backslash B_{R}.
 \end{split}
 \right.
\end{equation*}
Note that \begin{equation}\label{6}
 g(t,x,\hat{u})-g(t,x,\bar{u}_g)+\alpha\phi^+=[\int_0^1\frac{\p g}{\p u}(t,x,\bar{u}_g+\theta\phi^+)d\theta+\alpha]\cdot\phi^+.
\end{equation}
Since $|\bar{u}_g(t,x)|\leq\epsilon_*$ and
$|\theta\phi^+|\leq|\phi^+|\leq3\epsilon_*$ on $\R^n\backslash
B_{R}$, one has $|\bar{u}_g+\theta\phi^+|\leq \epsilon_0$. Thus by
(f 3) (with $f$ replaced by $g$), $\int_0^1\frac{\p g}{\p
u}(t,x,\bar{u}_g+\theta\phi^+)d\theta\leq -\alpha$. Note also that
$\phi^+>0$ on $\R^n\backslash\overline{B_R}$. It follows from
(\ref{6}) that $g(t,x,\hat{u})\leq g(t,x,\bar{u}_g)-\alpha\phi^+,$
which implies that
\begin{equation*}
\left\{\begin{split} &\dfrac{\partial \hat{u}}{\partial t}\geq
\Delta \hat{u} +g(t,x,\hat{u}), \quad x\in \R^n\backslash
\overline{B_{R}},
\, t>0,\\
&\hat{u}=\,3\epsilon_*+\bar{u}_g\geq2\epsilon_*,\quad\quad
\quad\quad x\in \p B_{R}, \, t>0,
\\
&\hat{u}(0,x)=3\epsilon_*+\bar{u}_g(0,x)\geq2\epsilon_*,
\quad\quad\quad x\in \R^n\backslash B_{R}.
 \end{split}
 \right.
\end{equation*} Combined with (\ref{3}) and (\ref{IBVP-sys}), the
comparison principle implies that
$$u(t,x,v_0,g)\leq \bar{u}_g(t,x)+\phi^+(t,x),~~\forall t\geq0,x\in\R^n\backslash B_{R}.$$
Similarly, we can construct $\phi^-$ satisfying
\begin{equation*}
\left\{\begin{split} &\dfrac{\partial \phi^-}{\partial t}= \Delta
\phi^- -\alpha\phi^-, \quad x\in \R^n\backslash \overline{B_{R}},
\, t>0,\\
&\phi^-=\,-3\epsilon_*,\quad\quad \quad\quad x\in \p B_{R}, \, t>0,
\\
&\phi^-(0,x)=-3\epsilon_*, \quad\quad\quad x\in \R^n\backslash
B_{R}.
 \end{split}
 \right.
\end{equation*} and obtain that
$$u(t,x,v_0,g)\geq \bar{u}_g(t,x)+\phi^-(t,x),~~\forall t\geq0,x\in\R^n\backslash B_{R}.$$
A direct estimate yields that (see \cite[P.94]{Ma85})
$$\lim_{\substack{t\rightarrow+\infty\\|x|\rightarrow+\infty}}\phi^{\pm}(t,x)=0,$$
which implies that
\begin{equation}\label{7}\lim_{\substack{t\rightarrow+\infty\\|x|\rightarrow+\infty}}|u(t,x,v_0,g)-\bar{u}_g(t,x)|=0.
\end{equation}

In order to prove the relative compactness of
$\{u(t,\cdot,v_0,g)\}_{t\in[0,\infty)}$ in $X$, we note that, by
(\ref{2})-(\ref{1}), $u(t,x,v_0,g)$ is a bounded solution of
($\ref{IVP-sys}_g$) in $X$. Then the standard parabolic estimate
shows that  $u(t,\cdot,v_0,g)$ is bounded in $C^{2}_{loc}(\R^n)$.
Combining (\ref{1}), (\ref{7}) and the Arzel\`{a}-Ascoli Theorem, we
 obtain the relative compactness of
$\{u(t,\cdot,v_0,g)\}_{t\in[0,\infty)}$ in $X$.
% that for any $\varepsilon>0$,
%there exist some positive numbers $R_0>0$ and $T>0$ such that
%\begin{equation}\label{un-conver}|v(t,x,v_0,g)-v(s,x,v_0,g)|<\varepsilon,~\mbox{ for
%}|x|\geq R_0 \mbox{ and } t,s\geq T.
%\end{equation}
%For such $\varepsilon>0$, let $B_{2R_0}=\{x\in\R^n:|x|< 2R_0\}$.
%Again by (\ref{2})-(\ref{1}), the standard parabolic estimate shows that  $\{v(t,\cdot,v_0,g)\}_{t\geq
%T}$ is bounded in $C^{2}_{loc}(\R^n)$. In particular,
%$\{v(t,\cdot,v_0,g)\}_{t\geq T}$ is bounded in
%$C^{2}(\overline{B_{2R_0}})$. Then Arzel\`{a}-Ascoli Theorem
%implies that there exists a sequence
%$\{t_n\}\rightarrow\infty$ such that
%$\{v(t_n,\cdot)\}_{n=1}^{\infty}$ is uniformly convergent in
%$C(\overline{B_{2R_0}})$, i.e.,
%\begin{equation}\label{in-conver} |v(t_n,x,v_0,g)-v(t_m,x,v_0,g)|<\varepsilon,~\mbox{ for
%all }|x|\leq 2R_0 \mbox{ and } t_n,t_m\geq T.
%\end{equation}
%Combing with \eqref{un-conver} and \eqref{in-conver}, we immediately obtain that
%$$|v(t_n,x,v_0,g)-v(t_m,x,v_0,g)|<\varepsilon,~\mbox{ for all }x\in\R^n \mbox{
%and } t_n,t_m\geq T.$$  In view of the completeness of $L^\infty(\R^n)$, one
%has that $\{v(t_n,\cdot,v_0,g)\}_{n=1}^{\infty}$ is convergent in
%$L^\infty(\R^n)$.
\end{proof}

\subsection{Traveling waves}
In this subsection, we will utilize the abstract results in Section 3 to investigate the monotonicity
of stable traveling waves for time-almost periodic reaction-diffusion equations with bistable nonlinearities.
Our aim is to study such kind of problems from a general point of view. As a simple illustrated example, we
consider the following time-almost periodic reaction-diffusion equation of the form:
\begin{equation}\label{equ.4.1}
\frac{\partial u}{\partial t}=\frac{\partial^2 u}{\partial
z^2}+f(t,u),\qquad z\in\mathbb{R},\,t>0,
\end{equation}
where the nonlinearity
$f(t,u):\mathbb{R}\times\mathbb{R}\rightarrow\mathbb{R}$ is a
$C^1$-admissible and uniformly almost periodic in $t$, real-valued
function. Of course, we remark that our approach for \eqref{equ.4.1} here can be applicable,
    with little modification, to monotonicity
of stable traveling waves for other various types of equations (see,
e.g. \cite{OM1,OM2}) with bistable nonlinearities.

A solution $u(z,t)$ of (\ref{equ.4.1}) is called an {\it almost
periodic traveling wave} (see, e.g. \cite[Section 2.2]{Sh1}), if
there are $\phi\in C^1(\R\times\R,\R)$ and $c\in C^1(\R,\R)$ such
that
$$u(z,t)=\phi(z-c(t),t),$$ where $\phi(x,t)$ (called {\it the wave
profile}) is almost periodic in $t$ uniformly with respect to $x$ in
bounded sets, and $c'(t)$ (called {\it the wave speed}) is almost periodic
in $t$;  and moreover, the frequency modules
$$\mathcal{M}(\phi(x,\cdot)),\quad\mathcal{M}(c'(\cdot))\subset \mathcal{M}(f).$$
We restrict our attention to traveling waves satisfying the connecting condition
\begin{equation*}\lim\limits_{x\rightarrow\pm\infty}\phi(x,t)=u^{f}_{\pm}(t),\qquad\mbox{
uniformly for } t\in\mathbb{R},
\end{equation*} where $u^{f}_{\pm}(t)$ are spatially homogeneous
time-almost periodic solutions of (\ref{equ.4.1}) with
$\mathcal{M}(u^{f}_{\pm}(\cdot))\subset \mathcal{M}(f).$ A traveling
wave is called a solitary wave if $u^f_+(t)=u^f_-(t)$ for all $t\in
\RR$, a traveling front if $u^f_-(t)<u^f_+(t)$ for all $t\in \RR$,
or $u^f_-(t)>u^f_+(t)$ for all $t\in \RR$.

In what follows we assume that

\vskip 2mm
{\bf (F)} ~~there exist an $\epsilon_0>0$ and a $\mu>0$ such that
$$\dfrac{\p f}{\p u}(t,u)\leq-\mu, \qquad \mbox{ for
}|u-u^{f}_{\pm}(t)|<\epsilon_0 ~\mbox{ and } t\in\mathbb{R}.
$$
\vskip 2mm

Let $X=C_{unif}(\mathbb{R})$ denote the space of bounded and
uniformly continuous functions on $\mathbb{R}$ endowed with the
$L^{\infty}(\mathbb{R})$ topology. For any $u_0\in X$, let
$u(\cdot,t;u_0,f)$ be the solution of (\ref{equ.4.1}) with
$u(\cdot,0;u_0,f)=u_0$.

A traveling wave $\phi(z-c(t),t)$ of (\ref{equ.4.1}) is called {\it
uniformly stable} if for every $\varepsilon>0$ there is a
$\delta(\varepsilon)>0$ such that, for every $u_0\in X$, if $s\geq0$
and
$\norm{u(\cdot,s;u_0,f)-\phi(\cdot-c(s),s)}_{L^\infty}\leq\delta(\varepsilon)$
then
$$\norm{u(\cdot,t;u_0,f)-\phi(\cdot-c(t),t)}_{L^\infty}<\varepsilon~\mbox{ for each }~t\geq s.$$
Moreover, $\phi(z-c(t),t)$ is called {\it uniformly stable with asymptotic phase} if it is
uniformly stable and there exists a $\delta>0$ such that if
$\norm{u_0-\phi(\cdot-c(0),0)}_{L^\infty}<\delta$ then
$$\norm{u(\cdot,t;u_0,f)-\phi(\cdot-c(t)-\sigma,t)}_{L^\infty}\rightarrow0~\mbox{ as }~t\rightarrow\infty$$
for some $\sigma\in\mathbb{R}$. A traveling wave $\phi(z-c(t),t)$ is
called {\it spatially monotone} if $\phi(x,t)$ is a non-decreasing or non-increasing function of $x$ for every $t\in\mathbb{R}$.

Based on our main abstract results, Theorems B and D, in Section 3, we derive the following results:
\vskip 2mm

\begin{theorem}\label{theorem 4.1} Any uniformly stable traveling
wave of \eqref{equ.4.1} is spatially monotone. In particular, solitary waves of \eqref{equ.4.1} are not uniformly stable.
\end{theorem}

\begin{theorem}\label{theorem 4.2} Any uniformly stable
traveling wave of (\ref{equ.4.1}) is uniformly stable with
asymptotic phase.
\end{theorem}

\begin{remark}
\textnormal{A converse result to Theorem \ref{theorem 4.1}, i.e., spatially monotone time-almost periodic traveling waves are
uniformly stable, was first obtained by Shen \cite{Sh1}. In \cite{Sh2,Sh3}, she further proved
 the existence of such traveling wave. The same result as Theorem \ref{theorem 4.2} can also be found in
Shen \cite{Sh1}. Note that our approach (Theorem D) was introduced in a very general framework,
 and hence, it can be applied to wider classes of equations
with little modification.}
\end{remark}

\noindent {\it Proof of Theorems \ref{theorem 4.1} and \ref{theorem 4.2}}. We first
rewrite equation (\ref{equ.4.1}) with the moving
coordinate $x=z-c(t)$:
\begin{equation}\label{equ.4.2}
\frac{\partial u}{\partial t}=\frac{\partial^2 u}{\partial
x^2}+c'(t)\frac{\partial u}{\partial x}+f(t,u),\qquad\qquad
x\in\mathbb{R},\,t>0.
\end{equation} Obviously, $\phi(z-c(t),t)$ is an almost
periodic traveling wave of (\ref{equ.4.1}) if and only if
$\phi(x,t)$ is an almost periodic entire solution of
(\ref{equ.4.2}) satisfying $\mathcal{M}(\phi(x,\cdot))\subset
\mathcal{M}(f)$.  In the following, we rewrite $\phi(x,t)$ as $\phi^{y_0}(x,t)$, with $y_0=(c^\prime,f)$,
for the sake of completeness. Therefore, it is easy to see that
\begin{equation}\label{equ.4.3}\lim\limits_{x\rightarrow\pm\infty}\phi^{y_0}(x,t)=u^{f}_{\pm}(t),\qquad\mbox{
uniformly in } t\in\mathbb{R}.
\end{equation}

Let $Y=H(c',f)$ be the hull of the function $y_0=(c',f)$. By the
standard theory of reaction-diffusion systems (see, e.g.
\cite{Fri,Hen}), it follows that for every $v_0\in X$ and
$y=(d,g)\in Y$, the system
\begin{equation}\tag{\ref{equ.4.2}$_y$}
\frac{\partial u}{\partial t}=\frac{\partial^2 u}{\partial
x^2}+d(t)\frac{\partial u}{\partial x}+g(t,u),\qquad\qquad
x\in\mathbb{R},\,t>0
\end{equation} admits a (locally) unique regular solution $v(\cdot,t;v_0,y)$ in
$X$ with $v(\cdot,0;v_0,y)=v_0$. This solution also continuously
depends on $y\in Y$ and $v_0\in X$ (see, e.g. \cite[Sec.3.4]{Hen}).
Therefore, (\ref{equ.4.2}$_y$) induces a (local) skew-product
semiflow $\Pi$ on $X\times Y$ with
$$\Pi_t(v_0,y)=(v(\cdot,t;v_0,y),y\cdot t),\quad\forall (v_0,y)\in X\times Y,\,t\geq0.$$
We define an order relation in $X$ by $$u\leq v~\mbox{ if }~u(x)\leq
v(x),~\forall x\in\mathbb{R}.$$ Let
$G=\{a_\sigma:\sigma\in\mathbb{R}\}$ be the group of translations
\begin{equation*}a_\sigma:u(\cdot)\mapsto u(\cdot-\sigma)
\end{equation*} acting on the space $X$. Then (A1)-(A3) are fulfilled.

Note that $\phi^{y_0}(x,t)$ is an uniformly almost periodic solution of
(\ref{equ.4.2}) with $\mathcal{M}(\phi^{y_0}(x,\cdot))\subset
\mathcal{M}(f)=\mathcal{M}(y_0)$. So, the closure $K$ of the orbit
$\{(\phi^{y_0}(\cdot,t),y_0\cdot t):t\in \RR\}$ of $\Pi_t$ is
 a uniformly stable 1-cover of $Y$. As a consequence, $K$ can be written as
 $$K=\{(\phi^y(\cdot,0),y)\in X\times Y:y=(d,g)\in Y\},$$ where the map $y\mapsto \phi^y(\cdot,0)\in X$
 is continuous and satisfies $\phi^{y_0}(\cdot,t)=\phi(\cdot,t)$ and
 $\phi^{y\cdot
t}(\cdot,0)=\phi^y(\cdot,t)$ for all $y\in Y$ and $t\in \RR.$ By virtue of \eqref{equ.4.3}, it is not difficult to see that
\begin{equation}\label{equ.4.3*}\lim\limits_{x\rightarrow\pm\infty}\phi^{y}(x,t)=u^{g}_{\pm}(t),\quad\mbox{
uniformly for } y=(d,g)\in Y\mbox{ and } t\in\mathbb{R},
\end{equation}
where $\{(u_{\pm}^g(0),g)\in \mathbb{R}\times H(f):g\in H(f)\}$ is a
1-cover of $H(f)$ and satisfies
$u_{\pm}^{g\cdot t}(0)=u_{\pm}^g(t)$ for all $g\in H(f)$ and $t\in \RR.$
 Of course, one can also easily see that, for any
$g\in H(f)$, the function-pair $(g,u^{g}_{\pm}(t))$ also satisfies
the condition (F), i.e.,
\vskip 2mm
{\bf (F)$_g$:} ~~there exist an $\epsilon_0>0$ and a $\mu>0$ such that
$$\dfrac{\p g}{\p u}(t,u)\leq-\mu, \qquad \mbox{ for
}|u-u^{g}_{\pm}(t)|<\epsilon_0 ~\mbox{ and } t\in\mathbb{R}.
$$
\vskip 2mm

In order to apply Theorems B and D in Section 3, we have to check (A4) there. By virtue of \eqref{equ.4.3*}
 and the condition (F)$_g$ above,
(A4-i) can be shown by repeating an analogue of Proposition \ref{compactness-or}, with $\bar{u}_g$ replaced by
$\phi^y-u^{g}_{\pm}$ (see also the similar arguments in \cite[Lemma 5.6]{OM1}). We omit the detail here.

As for (A4-ii), we will deduce it from Proposition \ref{lemma 4.4}
below. Based on this, we can apply Theorem B to obtain that the
group orbit $GK$ of $K$ is a $1$-D subbundle of $X\times Y$. In
particular, fix $y_0\cdot t\in Y$, the fibre
$$GK_{y_0\cdot t}=G[\phi^{y_0\cdot t}(x,0)]=G[\phi^{y_0}(x,t)]=G[\phi(x,t)]=\{\phi(x-\sigma,t):\sigma\in \RR\}$$ is totally-ordered,
which implies that $\phi(x,t)$ is monotone in $x$ for every $t\in \RR$. Furthermore, it follows from
Theorem D that the traveling wave $\phi(z-c(t),t)$ is uniformly stable with
asymptotic phase. This completes the proof of Theorems \ref{theorem 4.1} and \ref{theorem 4.2}.
$\qquad\qquad\qquad\qquad\qquad\qquad\qquad\qquad\qquad\qquad\qquad\qquad\qquad\qquad\square$

\begin{proposition}\label{lemma 4.4}
Let $\epsilon_0$ be given in {\rm (F)}.
For $(u_0,y_0)\in X\times Y$, suppose
that the omega limit set $\mathcal{O}(u_0,y_0)$ exists and satisfies
\begin{equation}\label{linf-close}\norm{v(\cdot)-\phi^y(\cdot,0)}_{L^\infty}<\frac{\epsilon_0}{2}\quad\mbox{ for
all }(v,y)\in\mathcal{O}(u_0,y_0),\end{equation} as well as
\begin{equation}\label{h-com-s}
(v(x),y)\leq (\phi^y(x-h,0),y),~v(x)\not\equiv
\phi^y(x-h,0),\quad\forall
(v,y)\in\mathcal{O}(u_0,y_0),\,x\in\mathbb{R},\end{equation} for some
$h\in\mathbb{R}$. Then there exists some $\delta>0$ such that
$$(v(x),y)\leq (\phi^y(x-h-\sigma,0),y),~v(x)\not\equiv
\phi^y(x-h-\sigma,0)$$ for all $(v,y)\in\mathcal{O}(u_0,y_0)$,
$x\in\mathbb{R}$ and $|\sigma|<\delta$. The assertion remains true
if the inequality sign $\leq$ is replaced by $\geq$.
\end{proposition}
\begin{proof} We use the similar arguments in Proposition \ref{lemma 4.8}.
Let $\mu$, $\epsilon_0$ be such that (F) holds.  By
\eqref{equ.4.3*}, we have
$$\lim\limits_{x\rightarrow\pm\infty}\phi^{y}(x-h,0)=\lim\limits_{x\rightarrow\pm\infty}\phi^{y}(x,0)=u^{g}_{\pm}(0),\qquad\mbox{
uniformly for } y=(d,g)\in Y.$$ Thus there exist some $R',~R''>0$
such that
\begin{equation}\label{equ.4.11}|\phi^{y}(x,0)-u^{g}_{\pm}(0)|<\frac{\epsilon_0}{2}\quad\mbox{ for all
}|x|>R'\mbox{ and }y\in Y,
\end{equation}
as well as
\begin{equation}\label{equ.4.12}|\phi^{y}(x-h,0)-u^{g}_{\pm}(0)|<\frac{\epsilon_0}{2}\quad\mbox{ for all
}|x|>R''\mbox{ and }y\in Y.\end{equation} Let $R=\max\{R',R''\}$, In view of \eqref{linf-close}, it follows from
(\ref{equ.4.11}) that
\begin{equation}\tag{{\bf A$'$}}|v(x)-u^{g}_{\pm}(0)|<\epsilon_0\quad\mbox{ for all
}(v,y)\in\mathcal{O}(u_0,y_0)\mbox{ and }|x|>R.\end{equation}
Moreover, combined with (\ref{equ.4.12}), the continuity of the
translation-group action on $X$ implies that
%it is easy to see that for any $\epsilon>0$ there exists
%some $\delta_\epsilon>0$ such that, if $|\sigma|<\delta_\epsilon$
%then
%\begin{equation}\label{equ.4.14}|\phi^{y}(x-\sigma,0)-\phi^{y}(x,0)|<\epsilon\quad\mbox{ for all
%}x\in\mathbb{R}\mbox{ and }y\in H(c',f).
%\end{equation}
%In fact, for any $y\in H(c',f)$, it follows from $\phi^y(\cdot,0)\in
%C_{unif}(\mathbb{R})$ that there exists some $\delta_y>0$ such that,
%if $|\sigma|<\delta_y$ then
%$$|\phi^{y}(x-\sigma,0)-\phi^{y}(x,0)|<\frac{\epsilon}{3}\quad\mbox{ for all
%}x\in\mathbb{R}.$$ Since $\phi^y(\cdot,0)$ is continuous on $y$,
%there exists a neighborhood $V_y\subset H(c',f)$ of $y$ such that if
%$\tilde{y}\in V_y$ then
%$$|\phi^{\tilde{y}}(x,0)-\phi^{y}(x,0)|<\frac{\epsilon}{3}\quad\mbox{ for all
%}x\in\mathbb{R}.$$ Consequently we have that if $|\sigma|<\delta_y$
%then
%\begin{equation*}\begin{aligned}&|\phi^{\tilde{y}}(x-\sigma,0)-\phi^{\tilde{y}}(x,0)|\leq
%|\phi^{\tilde{y}}(x-\sigma,0)-\phi^{y}(x-\sigma,0)|+\\&|\phi^{y}(x-\sigma,0)-\phi^{y}(x,0)|+|\phi^{y}(x,0)-\phi^{\tilde{y}}(x,0)|<\epsilon
%\end{aligned}
%\end{equation*} for all $x\in\RR$ and $\tilde{y}\in V_y$. Thus (\ref{equ.4.15}) follows immediately from  the compactness
%of $H(c,f)$. So we have proved the claim.
there exists a $\delta_0>0$ such that if
$|\sigma|<\delta_0$ then
%\begin{equation}\label{equ.4.15}|\phi^{y}(x-h-\sigma,0)-\phi^{y}(x-h,0)|<\frac{\epsilon_0}{2}\quad\mbox{ for all
%}x\in\mathbb{R}\mbox{ and }y\in H(c',f).
%\end{equation} Then it follows from (\ref{equ.4.12}) and (\ref{equ.4.15}) that if
%$|\sigma|<\delta_0$ then
\begin{equation}\tag{{\bf B$'$}}|\phi^{y}(x-h-\sigma,0)-u^{g}_{\pm}(0)|<\epsilon_0,\quad\mbox{
for all }|x|>R\mbox{ and }y\in Y.\end{equation}
%Consequently,
%we have that if $|\sigma|<\delta_0$ then
%\begin{equation}\tag{B$'$}|\phi^{y}(x-h-\sigma,t)-u^{y}_{\pm}(t)|<\epsilon_0\quad\mbox{
%for all }|x|>R,~y\in Y \mbox{ and }t\in\R^+.
%\end{equation}

Due to the assumption \eqref{h-com-s},
the strong maximum
principle yields that
$$(v(x,t;v,y),y\cdot t)< (\phi^{y\cdot t}(x-h,0),y\cdot t),\quad\forall
(v,y)\in\mathcal{O}(u_0,y_0),\,x\in\mathbb{R},\,t>0.$$ By virtue of the
invariance of $\mathcal{O}(u_0,y_0)$, we get that
$$(v(x),y)< (\phi^y(x-h,0),y),\quad\forall
(v,y)\in\mathcal{O}(u_0,y_0),\,x\in\mathbb{R}.$$ Since $\mathcal{O}(u_0,g_0)$
is compact in $X\times Y$, %it follows that
%for any $(v,y)\in\mathcal{O}(u_0,y_0)$ there exists some
%$\varepsilon(v,y)$ such that,
%$$(v(x),y)<(\phi^y(x-h,0)-\varepsilon(v,y),y)\quad\mbox{for all }|x|\leq R.$$
it follows from the continuity of $\phi^y(\cdot,0)$ on $y$ that for
a sufficiently small $\tilde{\epsilon}>0$,
$$(v(x),y)<(\phi^y(x-h,0)-\tilde{\epsilon},y)\quad\mbox{for all }(v,y)\in\mathcal{O}(u_0,y_0)\mbox{ and }|x|\leq R.$$
So one can find a $\delta>0$
$(\delta\leq\delta_0)$ such that if $|\sigma|<\delta$ then
\begin{equation}\tag{{\bf C$'$}}(v(x),y)<(\phi^y(x-h-\sigma,0),y)\quad\mbox{for all }(v,y)\in\mathcal{O}(u_0,y_0)\mbox{ and }|x|\leq R.
\end{equation}
%Indeed, suppose that (B) is not true. Then for any
%$\frac{1}{n}~(n\in\mathbb{N})$, there exist some $\sigma_n$ with
%$|\sigma_n|<\frac{1}{n}$, $|x_n|\leq R$ and
%$(v_n,y_n)\in\mathcal{O}(u_0,y_0)$ such that
%\begin{equation}\label{equ.4.17}(v_n(x_n),y_n)\geq(\phi^{y_n}(x_n-h-\sigma_n,0),y_n).\end{equation} Since
%$\mathcal{O}(u_0,g_0)$ is compact, let
%$$(v_n,y_n)\rightarrow(\tilde{v},\tilde{y})\in\mathcal{O}(u_0,g_0)\mbox{ and }x_n\rightarrow \tilde{x}\mbox{ with }|\tilde{x}|\leq R\quad\mbox{ as }n\rightarrow\infty.$$
%Then by the continuity of $\phi^y(\cdot,0)$ on $y$ and
%(\ref{equ.4.17}), we obtain
%$$(\tilde{v}(\tilde{x}),\tilde{y})\geq (\phi^{\tilde{y}}(\tilde{x}-h,0),\tilde{y}),$$ a contradiction. Thus we have
%proved (B).
Note also that
$\phi^{y}(x-h-\sigma,0)-v(x)\geq
\phi^{y}(x-h-\sigma,0)-\phi^{y}(x-h,0),\quad\forall
(v,y)\in\mathcal{O}(u_0,y_0),\,x\in\mathbb{R}.$ Then
\begin{equation}\tag{{\bf D$'$}}
\liminf\limits_{|x|\rightarrow\infty}(\phi^{y}(x-h-\sigma,0)-v(x))\geq\liminf\limits_{|x|\rightarrow\infty}(\phi^{y}(x-h-\sigma,0)-\phi^{y}(x-h,0))=0
\end{equation} for all $(v,y)\in\mathcal{O}(u_0,y_0)$ and $|\sigma|<\delta$.

Similarly as (A)-(D) in the proof of Proposition \ref{lemma 4.8}, we
can deduce from (A$'$)-(D$'$) that, for any
$(v,y)\in\mathcal{O}(u_0,y_0)$ and $\tau>0$, there exists some
$(v_{-\tau},y_{-\tau})\in\mathcal{O}(u_0,y_0)$ with
$\Pi_\tau(v_{-\tau},y_{-\tau})=(v,y).$ Moreover, for any
$|\sigma|<\delta$, the following statements hold true:
\begin{equation*}\begin{aligned}{\rm (i)}\quad &|v(x,t;v_{-\tau},y_{-\tau})-u^{g_{-\tau}\cdot t}_{\pm}(0)|<\epsilon_0\quad\mbox{ for all
}t>0\mbox{ and
}|x|>R,\\{\rm (ii)}\quad &|\phi^{y}(x-h-\sigma,t)-u^{g}_{\pm}(t)|<\epsilon_0\quad\mbox{
for all }|x|>R,~y\in Y\mbox{ and
}t\in\R^+,\\
{\rm (iii)}\quad &v(x,t;v_{-\tau},y_{-\tau})<\phi^{y_{-\tau}\cdot
t}(x-h-\sigma,0)\quad\mbox{ for all }t>0\mbox{ and }|x|\leq R,\mbox{
and }\\
{\rm (iv)}\quad &\liminf\limits_{|x|\rightarrow\infty}(\phi^{y_{-\tau}\cdot
t}(x-h-\sigma,0)-v(x,t;v_{-\tau},y_{-\tau}))\geq 0 \quad\mbox{ for
all }t>0.\end{aligned}
\end{equation*} Therefore, by using an analogue of the last paragraph in the proof of Proposition \ref{lemma 4.8} (The
proof of this modified version of Lemma \ref{lemma 4.9} is almost identical to that of Lemma \ref{lemma 4.9}),   we obtain that
\begin{equation*}\phi^{y_{-\tau}\cdot t}(x-h-\sigma,0)-v(x,t;v_{-\tau},y_{-\tau})\geq-2\epsilon_0 e^{-\mu t}\quad\mbox{for all }|x|>R\mbox{ and
}t>0.
\end{equation*}
In particular, by letting $t=\tau$,
\begin{equation*}\phi^{y}(x-h-\sigma,0)-v(x)=\phi^{y_{-\tau}\cdot \tau}(x-h-\sigma,0)-v(x,\tau;v_{-\tau},y_{-\tau})\geq-2\epsilon_0 e^{-\mu \tau},~\forall
~|x|>R.
\end{equation*}  Since $\tau>0$ is arbitrarily chosen, by letting
$\tau\rightarrow\infty$ we have that $$\phi^{y}(x-h-\sigma,0)\geq
v(x)$$ for all $|x|>R$, $(v,y)\in\mathcal{O}(u_0,y_0)$ and
$|\sigma|<\delta$. Note also (C$'$). We have proved the Proposition.
\end{proof}

%======================thebibliography=============================

\end{document}